\documentclass[letterpaper,11pt]{article}

\usepackage[utf8]{inputenc}
\usepackage[english]{babel}

\usepackage{latexsym}
\usepackage{amssymb}
\usepackage{amsthm}
\usepackage{amsmath}
\usepackage{mathrsfs}
\usepackage{amstext}
\usepackage{graphicx}  
\usepackage{amsfonts}
\usepackage{wrapfig}
\usepackage{sidecap} 
\usepackage{latexsym}
\usepackage{pdfpages}
\usepackage{epic}
\usepackage{fullpage}
\usepackage{color}
\usepackage{curves}
\usepackage{subfigure}
\usepackage{setspace}
\usepackage{amsmath}
\numberwithin{equation}{section}
\usepackage{anysize}
\usepackage{rotating}
\usepackage{enumerate} 
\usepackage{hyperref}
\usepackage{bbm}
\usepackage{graphicx}
\usepackage{subfigure}
\usepackage{algorithm}
\usepackage[noend]{algpseudocode}

\usepackage{bm}

\newtheorem{theorem}{Theorem}[section]

\newtheorem{proposition}[theorem]{Proposition}

\newtheorem{remark}[theorem]{Remark}

\DeclareMathOperator*{\argmin}{arg\,min}
\def\Om{\Omega}
\def\om{\omega}

\def\ds{\displaystyle}

\def\Fin{\hfill$\Box$}

\renewcommand{\leq}{\leqslant}
\renewcommand{\geq}{\geqslant}
\numberwithin{equation}{section}

\newcommand{\pt}{\partial_t}

\newcommand{\pnu}{\partial_\nu}

\title{\bf Numerical null controllability of parabolic PDEs \\ using Lagrangian methods}

\author{
	Enrique Fern\'andez-Cara\thanks{EDAN and IMUS, Universidad de Sevilla, Apartamento 1160, 41080 Sevilla, Spain.}
	\and 
	Roberto Morales\thanks{IMUS, Universidad de Sevilla, Apartamento 1160, 41080 Sevilla, Spain. 
	}
	\and 
	Diego A. Souza\thanks{EDAN, Universidad de Sevilla, Apartamento 1160, 41080 Sevilla, Spain.}
	}
\date{}
	
\begin{document}
\maketitle

\begin{abstract}
	In this paper, we study several theoretical and numerical questions concerning the null controllability problems for linear parabolic equations and systems for several dimensions.
	The control is distributed and acts on a small subset of the domain.
	The main goal is to compute numerically a control that drives a numerical approximation of the state from prescribed initial data exactly to zero.	
	We introduce a methodology for solving numerical controllability problems that is new in some sense. The main idea is to apply 
	classical Lagrangian and Augmented Lagrangian techniques
	to suitable constrained extremal formulations that involve unbounded weights in time 
	that make global Carleman inequalities possible. 
	The theoretical results are validated by satisfactory numerical experiments for spatially $2$D and $3$D problems. 
\end{abstract}

\tableofcontents

\section{Introduction}\label{Sec-1}

   Let~$\Omega\subset \mathbb{R}^d$ ($d\geq 1$) be a bounded domain with boundary~$\partial \Omega$ of class~$C^2$ and let~$T>0$ be given.
   Let us set~$Q:=\Omega\times (0,T)$ and~$\Sigma:=\partial\Omega\times (0,T)$.
   Also, for any open set~$\omega\subset \Omega$, we will put~$Q_\omega :=\omega\times (0,T)$.

   Let~$a\in W^{1,\infty}(Q)$ and~$b\in L^\infty(Q)$ be given.
   We will suppose that~$a$ satisfies~$a(x,t)\geq a_0>0$ in~$Q$ and, for convenience, we consider  the differential operators~$L$ and~$L^*$, with
   \begin{align*}
	Ly=\pt y-\textrm{div}(a\nabla y)+by\quad \text{and}\quad L^*z=-\pt z-\text{div}(a\nabla z)+bz.
   \end{align*}

   In this paper, we will mainly deal with distributed control problems for linear state systems of the form
   \begin{align}
	\label{eq:main:01}
	\begin{cases}
		Ly=v\mathbbm{1}_\omega &\textrm{ in }Q,\\
		y=0 &\text{ on }\Sigma,\\
		y(\cdot\,,0)=y_0 &\textrm{ in }\Omega,
	\end{cases}
   \end{align}
where~$\omega \subset \Om$ is a (maybe small) non-empty open set, $\mathbbm{1}_\omega$ is the associated characteristic function, $v \in L^2(Q_\om)$ and~$y_0\in L^2(\Omega)$.
   Here, $v = v(x,t)$ is the control and~$y=y(x,t)$ is the state. 

   It is well known that, for every~$v\in L^2(Q_\omega)$, there exists a unique weak solution~$y$ to~\eqref{eq:main:01}, with 
   \begin{align*}
	y\in C^0([0,T];L^2(\Omega)) \cap L^2(0,T;H_0^1(\Omega)).
   \end{align*}  

   The null controllability problem for~\eqref{eq:main:01} can be formulated as follows:
   for each~$y_0\in L^2(\Omega)$, find a control~$v\in L^2(Q_\omega)$ such that the corresponding solution to~\eqref{eq:main:01} satisfies
   \begin{align}
	\label{cond:null:controllability}
	y(x,T)=0,\quad x\in \Omega.
   \end{align}

   In the sequel, we will respectively denote by~$(\cdot\,,\cdot)$ and~$\| \cdot \|$ the usual scalar product and norm in~$L^2(\Omega)$ and the symbol~$C$ will stand for a generic positive constant.

\subsection{Literature}

   Controllability issues for PDEs have attracted the attention of the scientific community since the 80's; we mention~\cite{Lasiecka2000ControlI, Lasiecka2000ControlII, Russel1978Controllability, Coron2007Control} for a general overview.
   In particular, for the null controllability problem for the heat equation, we mention~\cite{Fursikov1996Controllability} and~\cite{Lebeau1995Controle}, where different approaches have been shown;
   see also~\cite{EFC2006Global}. 

   The regularity properties of the trajectories play an important role in the field.
   Thus, in the case of a parabolic equation, the regularization effect leads to some nontrivial difficulties that were exhibited numerically for the first time in~\cite{Carthel1994On}, where the authors try to find controls of minimal~$L^2$ norm. 

   Let us introduce two functions~$\rho$ and~$\rho_0$ with
   \begin{align}
	\label{assumptions:0:rho}
	\begin{cases}
		\rho=\rho(x,t),\,\rho_0=\rho_0(x,t)\textrm{ are continuous and }\geq \rho_*>0\textrm{ in }Q,
		\\ \ds
		\min_{x\in \overline{\Omega}} \rho(x,t)\to +\infty \text{ and }\min_{x\in \overline{\Omega}} \rho_0(x,t) \to +\infty \text{ as }t\to T^-
	\end{cases}
   \end{align}
and let us set
\[
	\mathcal{U}:=\rho_0^{-1}L^2(Q_\omega)=\{v\in L^2(Q_\omega) :  \rho_0 v\in L^2(Q_\omega) \},
\]
\[
	\mathcal{Y}:=\rho^{-1} L^2(Q)=\{y\in L^2(Q_\omega) :  \rho y\in L^2(Q) \}
\]
and
   \begin{align*}
	\mathcal{C}(y_0,T):=\{(y,v)\in \mathcal{Y} \times \mathcal{U} : y\textrm{ solves }\eqref{eq:main:01} \}.
   \end{align*}

   We will consider the constrained extremal problem
   \begin{align}
	\label{Extremal:problem:01}
	\begin{cases}
		\textrm{Minimize }\displaystyle J(y,v)=\dfrac{1}{2}\iint_{Q} \rho^2 |y|^2\,dx\,dt + \dfrac{1}{2}\iint_{Q_\omega} \rho_0^2 |v|^2 \,dx\,dt,\\
		\textrm{Subject to }(y,v)\in \mathcal{C}(y_0,T),
	\end{cases}
   \end{align}
introduced in the 90's by Fursikov and Imanuvilov;
   see~\cite{Fursikov1996Controllability} and the references therein.

\begin{remark}\label{rmk:weigths_NC}{\rm
   It is worth mentioning that, in~\eqref{Extremal:problem:01}, the null controllability requirement is implicitly imposed in the requirement~$(y,v) \in \mathcal{C}(y_0,T)$.
   In fact, it is known that~$\mathcal{C}(y_0,T)\neq \emptyset$ and any couple~$(y,v)\in \mathcal{C}(y_0,T)$ provides a solution to the null controllability problem~\eqref{eq:main:01}--\eqref{cond:null:controllability}.
   }\Fin
\end{remark}

   In this paper, we focus on the analysis and resolution of~\eqref{Extremal:problem:01} by using {\it Lagrangian} and {\it Augmented Lagrangian} methods.

   From the numerical controllability viewpoint, it is natural to consider the computation of minimal~$L^2$-norm null controls.
   This corresponds to~\eqref{Extremal:problem:01} with~$\rho\equiv 0$ and~$\rho_0\equiv 1$ and has been considered by~Carthel~et.~al.\ in~\cite{Carthel1994On} and then by other authors.
   
   The solution can be achieved as follows. For every~$\varphi_0 \in L^2(\Om)$, we consider the associated backward system
   \begin{align}
	\label{adjoint:varphi}
	\begin{cases}
		L^*\varphi =0&\textrm{ in }Q,\\
		\varphi=0&\text{ on }\Sigma,\\
		\varphi(\cdot\,,T)=\varphi_0&\text{ in }\Omega .
	\end{cases}
   \end{align}  
   Then, the null control of minimal norm in~$L^2(Q_\om)$ is given by~$v = \hat{\varphi}\mathbbm{1}_\om$, where~$\hat{\varphi}$ is the solution to~\eqref{adjoint:varphi} that corresponds to~$\hat{\varphi}_0$ and~$\hat{\varphi}_0$, minimizes the functional
   \begin{align*}
	I(\varphi_0):=\dfrac{1}{2}\iint_{Q_\omega} |\varphi(x,t)|^2\,dx\,dt + \int_\Omega y_0(x) \varphi(x,0) \,dx,   
	\end{align*}
over the Hilbert space~$\mathcal{H}$, given by the completion of~$L^2(\Omega)$ with respect to the norm
   \begin{align*}
	\|\varphi_0\|_{\mathcal{H}}:=\|\varphi\|_{L^2(Q_\omega)}.
   \end{align*}

   Note that the mapping~$\varphi_0 \mapsto \|\varphi_0\|_{\mathcal{H}}$ is a semi-norm in~$\mathcal{D}(\Omega)$.
   In view of the unique continuation property satisfied by the solutions to the systems~\eqref{adjoint:varphi}, it is in fact a pre-Hilbertian norm.
   Hence, the completion of~$\mathcal{D}(\Omega)$ for this norm can be considered.

   Furthermore, one has the observability inequality  
   \begin{align}\label{1.6p}
	\|\varphi(\cdot\,,0)\|^2\leq C\iint_{Q_\omega} |\varphi|^2\,dx\,dt\quad \forall \varphi_0\in L^2(\Omega)
   \end{align}
and, consequently, the coerciveness of the functional~$I$ in~$\mathcal{H}$ is ensured.
   Note that~\eqref{1.6p} is a consequence of some appropriate global Carleman estimates, see for instance~\cite{Fursikov1996Controllability} and~\cite{EFC2006Global}. 

   As explained in~\cite{Munch2010Numerical}, one has~$H^{-s}(\Omega) \hookrightarrow \mathcal{H}$ for all~$s>0$ with continuous embedding. For a proof of this assertion in 1-D case, see \cite{Micu2011On}. Accordingly, the minimization of~$I$ is numerically ill-posed:
   it does not seem easy to find a family of finite-dimensional ordered spaces that approximate the functions~$\varphi_0$ in the sense of the~$\mathcal{H}$ norm as the dimension grows to infinity.
   For this reason, in~\cite{Carthel1994On} the authors considered regularized approximate controllability problems, replacing~$I$ by~$I_\epsilon$, where 
   \begin{align*}
	I_\epsilon(\varphi_T):=I(\varphi_T)+\epsilon \|\varphi_T\|_{L^2},\quad \epsilon>0.
   \end{align*}

   The minimizers~$\varphi_{T,\epsilon}\in L^2(\Omega)$ and therefore the corresponding controls~$v_\epsilon$ produce states~$y_\epsilon$ that satisfy~$\|y_\epsilon(\cdot\,,T)\|_{L^2}\leq \epsilon$.
   The main advantage of this approach is that it can handle 
various boundary conditions and can be adapted to different types of parabolic and hyperbolic equations.
   However, for small $\epsilon$, the controls obtained by this method oscillate near the final time, unless a careful approximation and/or penalization process is performed;
   see~\cite{boyer2,  boyer1, Carthel1994On, Glowinski1995Exact, Glowinski2008Exact, Munch2010Numerical} for more details.



   An alternative is the so called {\it flatness approach.}
   This is a direct method for which the solution relies on the computation of the sums of appropriate series expansions.
   The corresponding partial sums are easy to compute and provide accurate numerical approximations of both the control and the state.
   The main advantage of this method is that it provides explicit control laws for certain problems.
   However, its implementation for high-dimensional systems can be cumbersome and the requirement to the system to be flat (i.e.\ to have time independent coefficients) makes the scope of the method limited;
   see~\cite{flatness1, flatness2, flatness3} for more details.

%

   A second alternative is given by the space-time strategy introduced in~\cite{FC2013Strong}.
   It relies on the  Fursikov-Imanuvilov formulation of controllability problems, where a weighted integral involving  both variables, state and control, is introduced, see~\cite{Fursikov1996Controllability} and the references therein;
   see also~\cite{EFC2006Global}.
   
   In this method, the task is reduced to solve a second-order in time and fourth-order in space PDE system.
   The main advantage is that  the strong convergence of the approximations is obtained from C\'ea's Lemma and a good choice of the associated finite dimensional approximation spaces.
   However, the numerical solution via a direct method requires in practice~$C^0$ in time and~$C^1$ in space finite elements and consequently is not easy to handle in high dimension.
   
   The~$C^1$ regularity drawback can be circumvented by introducing mixed formulations but this is not to our knowledge completely well-justified from the theoretical viewpoint.
   Moreover, for high dimension problems, the method is far from simple;
   for instance, to solve a control problem for~$3$D heat equation, at least~$\mathbb{P}_1$-Lagrange~$4$D finite elements are needed.

\subsection{Plan of the paper}

	In this paper, we will apply other methods that also start from the space-time strategy. 
	The idea is to introduce some Lagrangian (saddle-point) reformulations of~\eqref{Extremal:problem:01}.
	The techniques are relatively well known in many other contexts and have been applied since several decades to various PDE problems, see \cite{Fortin1983augmented}.
	
	Let us mention some advantages of the use of {\it Lagrangian} and {\it Augmented Lagrangian} formulations in the context of controllability problems:
	
\begin{itemize}
	
	\item They are easy to adapt to parabolic problems with nonzeo right-hand sides with a suitable exponential decay as~$t\to T^-$. 
		
	\item Various kinds of boundary conditions can be considered.		
			
	\item The well-posedness of the numerical approximations can be rigorously justified. 
			
	\item The methods are useful for high dimension problems, where there are very few efficient techniques. 
	Thus, the implementation for~$3$D parabolic equations is similar to the~$2$D case, and does not bring extra difficulties. 
		
	\item The methods are compatible (and improvable) with adaptive mesh refinement.
	
\end{itemize}

   The  paper is organized as follows.
   
   Section~\ref{Section:study:extremal:problem:trunctation} is devoted to analyze the extemal problem~\eqref{Extremal:problem:01}, a family of ``truncated'' approximations and their properties.
   In~Section~\ref{Section:study:extremal:problem}, we introduce Lagrangian and Augmented Lagrangian (saddle-point) reformulations of he truncated problems.
   Then, several related algorithms of the Uzawa kind are presented and described in~Section~\ref{Section:iterative:algorithms}.
   Section~\ref{Section:Stokes} is devoted to the null controllability problem for the Stokes system and its Lagrangian reformulations.
   In~Section~\ref{Section:numerical:experiments}, we present with detail the results of several numerical experiments for $2$D and $3$D problems concerning the heat equation and the Stokes system.
   Finally, Section~\ref{Section:further:comments} contains some conclusions, additional comments and open questions.   
   
%
%



\section{Truncation of the extremal control problem}
\label{Section:study:extremal:problem:trunctation}

   In order to formulate with detail our first main problem, we consider a function~$\eta^0\in C^2(\overline{\Omega})$ such that  
   \begin{align}\label{weight:conditions}
\eta^0=0\text{ on }\partial\Omega,\quad |\nabla \eta^0|>0\text{ in }\Omega\setminus \overline{\omega}\quad \text{and}\quad \pnu \eta^0 <0\text{ on }\partial\Omega.
   \end{align}
   
   The existence of~$\eta^0$ is proved in~\cite{Fursikov1996Controllability} for~$\Omega$ of class~$C^2$ and arbitrary~$\omega$.
   As we will see in~Section~\ref{Section:numerical:experiments}, it is also possible to construct explicit functions~$\eta^0$ when~$\Omega$ is a bounded domain with polygonal or polyhedrical boundary. 

   Let us set 
   \begin{align}\label{Carleman:weights}
\alpha(x,t):=K_1 \dfrac{e^{K_2}-e^{\eta^0(x)}}{T-t},\quad 
\rho(x,t):=e^{ \alpha(x,t)}\quad \text{and}\quad \rho_0(x,t):=(T-t)^{3/2}\rho(x,t),
   \end{align}
where~$K_1,K_2>0$ are sufficiently large.
   It is immediate to check that these functions~$\rho$ and~$\rho_0$ satisfy~\eqref{assumptions:0:rho}. 

   Let us introduce the linear operator~$M:L^2(Q_\omega)\to L^2(Q)$ with~$Mv:=y$, where~$y$ is the weak solution to the linear problem 
   \begin{align*}
	\begin{cases}
		Ly= v\mathbbm{1}_\omega&\text{ in }Q,\\
		y=0&\text{ on }\Sigma,\\
		y(\cdot\,,0)=0&\text{ in }\Omega.
	\end{cases}
   \end{align*}
   Its adjoint~$M^*:L^2(Q)\to L^2(Q_\omega)$ is given by~$M^*q=z|_{Q_\omega}$, where~$z$ is the weak solution to
   \begin{align*}
	\begin{cases}
		L^*z=q&\text{ in }Q,\\
		z=0&\text{ on }\Sigma,\\
		z(\cdot\,,T)=0&\text{ in }\Omega.
	\end{cases}
   \end{align*}

   In the sequel, we denote by~$\overline{y}$ the solution to~\eqref{eq:main:01} with~$v \equiv 0$.
   Then the set~$\mathcal{C}(y_0,T)$ in~\eqref{Extremal:problem:01} can be written in the form
   \begin{align*}
	\mathcal{C}(y_0,T)=\{(y,v)\in \mathcal{Y}\times \mathcal{U} :  y=Mv+\overline{y} \}.
   \end{align*}

   Let us introduce the linear space
   \begin{align*}
	P_0:=\{q\in C^2(\overline{Q}) : q=0\text{ on }\Sigma \}.
   \end{align*}
   In this space, the bilinear form 
   \begin{align*}
	(p,q)_P:=\iint_Q \rho^{-2}L^*p\,L^*q\,dx\,dt + \iint_{Q_\omega} \rho_0^{-2} p\,q\,dx\,dt
   \end{align*}
	is in fact a scalar product.
   Indeed, if we have~$q\in P_0$, $L^*q=0$ in~$Q$ and~$q=0$ in~$Q_\omega$, from the unique continuation property for parabolic equations, we deduce that~$q\equiv 0$.
   
   Let~$P$ be the completion of~$P_0$ for this scalar product (a Hilbert space).
   We have the following result:
   
\begin{proposition}\label{proposition:existence:minimizer}
   Let~$\rho,\rho_0$ be the functions defined in~\eqref{Carleman:weights}.
   Then~\eqref{Extremal:problem:01} has a unique minimizer~$(y,v)\in \mathcal{Y}\times \mathcal{U}$.
   Moreover, we have the following characterization: 
   \begin{align*}
		y=-\rho^{-2}q,\quad v=\rho_0^{-2}M^*q.
		y=-\rho^{-2}L^*p,\quad v=\rho_0^{-2} p\big|_{Q_\omega} ,
    \end{align*}
    where~$p\in P$ is the unique solution to the variational (Lax-Milgram-like) problem
    \begin{align}
	        \label{FI-LM}
 		\begin{cases}
 			\displaystyle \iint_{Q} \rho^{-2}L^*p\,L^*q\,dx\,dt + \iint_{Q_\omega} \rho_0^{-2}p\,q\,dx\,dt=\int_\Omega y_0(x)\,q(x,0)\,dx,\\
 			\forall q\in P,\quad p\in P. 
 		\end{cases}
    \end{align} 
\end{proposition}

   The proof is given in~\cite{Fursikov1996Controllability};
   see also~\cite[Proposition~$2.1$]{FC2013Strong}.

\

\begin{remark}{\rm
	In this result, the hypothesis that the Carleman weights~$\rho$ and~$\rho_0$ are as in~\eqref{Carleman:weights} is crucial.
	Indeed, this is essential to prove the continuity of the linear form in the Lax-Milgram problem~\eqref{FI-LM}, thanks to a Carleman inequality.
	}\Fin  
\end{remark}

\begin{remark}{\rm
   Note that the solution to~\eqref{FI-LM} solves, at least in the weak sense, the following fourth-order in space and second-order in time problem:
   \begin{align*}
	\begin{cases}
		L(\rho^{-2}L^*p)+\rho_0^{-2} p\, \mathbbm{1}_\omega =0&\text{ in }Q,\\
		p=0,\quad \rho^{-2}L^*p=0&\text{ on }\Sigma,\\
		\rho^{-2}L^*p(\cdot\,,0)=y_0,\quad \rho^{-2}L^*p(\cdot\,,T)=0&\text{ in }\Omega;
	\end{cases}
   \end{align*}
see~\cite{Fursikov1996Controllability}.
   }\Fin
\end{remark}


\begin{remark}{\rm
	Note also that the solution to~\eqref{FI-LM} solves the following extremal problem:
	\begin{align}
	\label{Extremal:problem:p}
	\begin{cases}
		\textrm{Minimize }\displaystyle J^*(q):=\dfrac{1}{2}\iint_{Q} \rho^{-2} |L^*q|^2\,dx\,dt + \dfrac{1}{2}\iint_{Q_\omega} \rho_0^{-2} |q|^2 \,dx\,dt -\int_\Omega y_0(x)q(x,0)\,dx,\\
		\textrm{Subject to }q\in P.
	\end{cases}
    \end{align}
	In some sense, \eqref{Extremal:problem:p} can be viewed as a dual problem of~\eqref{Extremal:problem:01}.
	}\Fin
\end{remark}

\subsection{A truncated problem}

   A detailed characterization of the minimizer of~\eqref{Extremal:problem:01} is furnished by~Proposition~\ref{proposition:existence:minimizer}.
   A related numerical treatment was performed in~\cite{FC2013Strong} by using direct and mixed variational approximations of~\eqref{FI-LM}.
   
   In order to apply Lagrangian and Augmented Lagrangian methods avoiding technical difficulties, let us introduce the problem
   \begin{align}
	\label{truncated:problem}
	\begin{cases}
		\textrm{Minimize }\displaystyle  J_R(y,v):=\dfrac{1}{2}\iint_{Q} \rho_R^2 |y|^2\,dx\,dt + \dfrac{1}{2}\iint_{Q_\omega} \rho_0^2 |v|^2 \,dx\,dt,\\
		\textrm{Subject to }(y,v)\in \widetilde{\mathcal{C}}(y_0,T),
	\end{cases}
   \end{align}
where~$R > 0$,
\begin{equation}\label{weight:rho_R}
	\rho_R(x,t) := \min(R,\rho(x,t)),\quad (x,t)\in Q
\end{equation} 
and
   \begin{align*}
	\widetilde{\mathcal{C}}(y_0,T):=\{(y,v)\in L^2(Q)\times \mathcal{U} : y\textrm{ solves }\eqref{eq:main:01} \}.
   \end{align*}
   
   We point out that the new weight~$\rho_R$ does not blow up as~$t\to T^-$.

   Recall that, for each~$y_0\in L^2(\Omega)$, the weak solution to~\eqref{eq:main:01} can be written in the form~$y=Mv+\overline{y},$ where~$M$ was introduced at the beginning of Section~\ref{Section:study:extremal:problem}.
   Accordingly, we can rewrite~\eqref{truncated:problem} in such a way that the results of~\cite{Ekeland1999Convex} concerning Lagrange methods can be applied.

   Thus, let us introduce the convex functions
     \[
F_R:L^2(Q)\to \mathbb{R}, \ \ G:\mathcal{U}\to \mathbb{R}\ \text{ and } \ \overline{F}_R:L^2(Q)\to \mathbb{R},
   \]
with  
   \begin{align*}
F_R(y):=\dfrac{1}{2}\iint_Q \rho_R^2 |y|^2\,dx\,dt,\quad G(v):=\dfrac{1}{2}\iint_{Q_\omega} \rho_0^2 |v|^2 dx\,dt\quad \text{and}\quad \overline{F}_R(z):=F_R(z+\overline{y}).
   \end{align*} 
   
   Observe that, for every~$(y,v)\in \widetilde{\mathcal{C}}(y_0,T)$, one has
   \begin{align*}
	J_R(y,v)=I_R(v):=\overline{F}_R(Mv)+G(v) \quad \forall \,v\in \mathcal{U}.
   \end{align*} 
   Here, $I_R:\mathcal{U}\to \mathbb{R}$ is a well-defined proper, continuous and strictly convex function and~\eqref{truncated:problem} is equivalent to the unconstrained extremal problem
   \begin{align}
	\label{extremal:problem:02}
	\begin{cases}
		\textrm{Minimize }I_R(v),\\
		\textrm{Subject to }v\in \mathcal{U}.
	\end{cases}
   \end{align}

   Obviously, for every $R>0$, \eqref{extremal:problem:02} is uniquely solvable.  
   Furthermore, the following result holds:

\begin{proposition}\label{Proposition:convergence}
   Let~$v_R$ be, for each $R > 0$, the unique solution to~\eqref{extremal:problem:02}.
   Then one has
   \begin{align}
		\label{convergence:vR:v}
		v_R\to \widehat{v}\text{ strongly in }\mathcal{U}\text{ as }R\to +\infty ,
   \end{align}
where~$\widehat{v}$ is, together with~$\widehat{y}=M\widehat{v}+\overline{y}$, the unique solution to~\eqref{Extremal:problem:01}.
\end{proposition}

\begin{proof}
   First, we note that, for any~$R > 0$,
   \begin{align*}
I_R(\widehat{v})=\dfrac{1}{2}\iint_Q \rho_R^2 |M\widehat{v}+\overline{y}|^2\,dx\,dt + \dfrac{1}{2}\iint_{Q_\omega} \rho_0^2 |\widehat{v}|^2\,dx\,dt\leq J(\widehat{y},\widehat{v}). 
   \end{align*}
   Consequently, each~$v_R$ satisfies
   \begin{align}\label{relation:vR:J:v}
		I_R(v_R)\leq I_R(\widehat{v})\leq  J(\widehat{y},\widehat{v})\quad \forall R>0.
   \end{align}
   This shows that~$\rho_R(Mv_R+\overline{y})$ is uniformly bounded in~$L^2(Q)$ and~$\rho_0 v_R$ is uniformly bounded in~$L^2(Q_\omega)$.
   Therefore, at least for a subnet, one has that there exist~$z \in L^2(Q)$ and~$w\in L^2(Q_\om)$ with
   \begin{align}\label{proof:convergence:01}
\rho_R(Mv_R+\overline{y})\to z\text{ weakly in }L^2(Q)
\ \text{ and } \ 
\rho_0 v_R\to w\text{ weakly in }L^2(Q_\omega) \text{ as $R\to +\infty$.}
   \end{align}
   
   Let us set~$\tilde{y}:=\rho^{-1}z$ and~$\tilde{v}:=\rho_0^{-1}w$.
   It is then clear, from~\eqref{proof:convergence:01} and {\it Lebesgue's Theorem,} that 
   \begin{align*}
v_R\to \tilde{v}\text{ weakly in }L^2(Q_\omega)
\ \text{ and } \ 
Mv_R+\overline{y}\to \tilde{y}\text{ weakly in }L^2(Q).
   \end{align*}
	
	In fact, $\tilde{y}$ is the state associated to~$\tilde{v}$ with initial condition~$y_0$ and~$Mv_R+\overline{y}$ converges strongly to~$\tilde{y}$ in~$L^2(Q)$, thanks to the usual parabolic compactness results.
	Furthermore, one has 
   \begin{align*}
J(\tilde{y},\tilde{v})\leq & \liminf_{R\to \infty} \left(\dfrac{1}{2} \iint_Q \rho_R^2 |Mv_R+\overline{y}|^2 \,dx\,dt 
+\dfrac{1}{2} \iint_{Q_\omega} \rho_0^2 |v_R|^2 \,dx\,dt\right)
\\ \leq & \lim_{R\to \infty} \left(\dfrac{1}{2} \iint_{Q}\rho_R^2 |Mv + \overline{y}|^2\,dx\,dt
+ \dfrac{1}{2} \iint_{Q_\omega} \rho_0^2 |v|^2\,dx\,dt \right) = J(y,v),
   \end{align*}
for every~$(y,v)\in \mathcal{Y}\times \mathcal{U}$. Hence, $(\tilde{y},\tilde{v})=(\widehat{y},\widehat{v})$.
	
	Finally, we also deduce from the properties satisfied by the~$v_R$ that 
   \begin{align*}
		\limsup_{R\to \infty} \left(\iint_{Q} \rho_R^2 |Mv_R+\overline{y}|^2 \,dx\,dt + \iint_{Q_\omega} \rho_0^2 |v_R|^2\,dx\,dt \right)\leq J(y,v)\quad \forall (y,v) \in \mathcal{Y}\times\mathcal{U},
   \end{align*}
whence this upper limit is bounded from above by~$J(\widehat{y},\widehat{v})$ and then
   \begin{align*}
		\lim_{R\to \infty} \left(\iint_{Q} \rho_R^2 |Mv_R+\overline{y}|^2 \,dx\,dt + \iint_{Q_\omega} \rho_0^2 |v_R|^2\,dx\,dt \right)= J(\widehat{y},\widehat{v}).
   \end{align*}
	Therefore, $v_R \to \widehat{v}$ strongly in~$\mathcal{U}$, $\rho_R(Mv_R+\overline{y}) \to \rho \widehat{y}$ strongly in~$L^2(Q)$ and~\eqref{convergence:vR:v} holds. 
\end{proof}

   As a consequence of this result, we see that an appropriate strategy to solve the null controllability problem for~\eqref{eq:main:01} can be to compute the solution to~\eqref{extremal:problem:02} for large~$R$ and then take~$y_R = Mv_R + \overline{y}$ as an approximation to~$\widehat{y}$.

 \subsection{Estimates of the convergence rate of the truncated solutions}

   The following result holds:
   
\begin{proposition}\label{prop:estimate:y:T}
   Let~$v_R$ be the solution to the extremal problem~\eqref{extremal:problem:02} and let us set~$y_R=Mv_R+\overline{y}$.
   There exists a positive constant~$C$ independent of~$R$ such that
   \begin{align}\label{estimate:y:T}
\|y_R(\cdot\,,T)\|_{H_0^1(\Omega)}\leq C\,\dfrac{|\log R|}{R} \quad \forall R > 0.
   \end{align}
\end{proposition}

\begin{proof}
   The proof relies on a suitable energy estimate of~$\rho_R^{-1}q$ at times close to~$T$, with a right-hand side that does not depend of~$R$.
   
   It can be assumed that~$R$ is large enough.
   Then, let us introduce~$\rho_0^*$ with
   \begin{align*}
\rho_0^*(t)=\min_{x\in \overline{\Omega}} \rho_0(x,t).
   \end{align*}	
   Denote by~$T_R$ the time at which~$\rho_0^*(T_R)=R$.
   Then, since~$\rho_0^*$ is increasing, we have 
   \begin{align}\label{rho:0:T:star}
\rho_0(x,t)\geq R \quad \forall x\in \Omega, \ \ \forall t\in [T_R,T].
   \end{align}

   Now, let~$T_{1,R},T_{2,R}\in (T_R,T)$ be given with~$T-T_{1,R}\leq1$ and~$T_{2,R}-T_{1,R} \geq (T - T_R)/2$.
   Then, let us consider a cut-off function~$\psi_R \in C^\infty([0,T])$ satisfying
   \begin{align*}
			0\leq \psi_R\leq 1\text{ in }[0,T],\quad 
			\psi_R=0\text{ in } [0,T_{1,R}] \quad \text{and} \quad
			\psi_R=1\text{ in }[T_{2,R},T]
   \end{align*}
and let us set~$z := \psi_R \,\rho_R \,y_R$.
   Then~$z$ solves the problem
   \begin{align*}
		\begin{cases}
			Lz= (\chi_R \rho_R) v_R \mathbbm{1}_\omega 
			+ \pt(\chi_R \rho_R )y_R  &\text{ in }Q,\\
			z=0&\text{ on }\Sigma,\\
			z(\cdot\,,0)=0&\text{ in }\Omega.
		\end{cases}
   \end{align*}
	
   Taking into account~\eqref{rho:0:T:star} and~\eqref{relation:vR:J:v}, we find that
   \begin{align*}
   \begin{array}{l} \ds
		\iint_{Q} |Lz|^2\,dx\,dt \leq \iint_{\om \times (T_{1,R},T)} \rho_0^2 |v_R|^2\,dx\,dt + \iint_{\om \times (T_{1,R}},{T_{2,R})}(\chi'_{R})^2\rho_R^2 |y_R|^2\,dx\,dt\\ 
	\noalign{\smallskip}\ds \phantom{\iint_{Q} |Lz|^2\,dx\,dt }	\leq
\iint_{\omega\times(0,T)} \rho_0^2 |v_R|^2\,dx\,dt
+ \frac{C}{(T_{2,R}-T_{1,R})^2} \iint_Q \rho_R^2 |y_R|^2\,dx\,dt \\ 
	\noalign{\smallskip} \ds \phantom{\iint_{Q} |Lz|^2\,dx\,dt }\leq \frac{C}{(T_{2,R}-T_{1,R})^2} J(\widehat{y},\widehat{v}),
   \end{array}
   \end{align*}
where~$\widehat{v}$ is the solution to~\eqref{Extremal:problem:01} and~$\widehat{y}=M\widehat{v}+\overline{y}$.
   Furthermore, from the particular form of~$\rho_0$, we now that~$T - T_R \leq C/|\log R|$ and this implies
   \begin{align*}
		\|z\|_{C^0(0,T];H_0^1(\Omega))}^2 + \|z\|_{L^2(0,T;H_0^1(\Omega)\cap H^2(\Omega))}^2 \leq \frac{C}{(T_{2,R}-T_{1,R})^2} J(\widehat{y},\widehat{v}) \leq C J(\widehat{y},\widehat{v}) |\log R|^2 ,
   \end{align*}
for some~$C>0$ independent of~$R$.
   Hence,
   \begin{align*}
		\|\rho_R^{-1}(\cdot\,,T)q_R(\cdot\,,T)\|_{H_0^1}\leq \|\rho_R^{-1}q_R\|_{C^0([T_{2,R},T];H_0^1(\Omega))}\leq C J(\widehat{y},\widehat{v})^{1/2} |\log R|
   \end{align*}
and, from the Cauchy-Scharwz inequality and the fact that~$\rho_R(\cdot\,,T) \equiv R$ in~$\Omega$, we find that
   \begin{align*}
\| \nabla y_R(\cdot\,,T) \| = \|\rho_R(\cdot\,,T)^{-1} \nabla (\rho_R(\cdot\,,T)q_R(\cdot\,,T))\|
\leq \dfrac{C\,|\log R|}{R} .
   \end{align*}

   This implies~\eqref{estimate:y:T} and ends the proof.
\end{proof}

\section{Solving the truncated extremal problem using Lagrange methods}
\label{Section:study:extremal:problem}

 In this section, we use more or less standard results from convex analysis to reformulate~\eqref{extremal:problem:02} appropriately and deduce related efficient algorithms.

\subsection{The Lagrangian}\label{Sec2.2}

   Let us introduce the convex functions
\[
F_R:L^2(Q)\to \mathbb{R}, \ \ G:\mathcal{U}\to \mathbb{R}\ \text{ and } \ \overline{F}_R:L^2(Q)\to \mathbb{R},
   \]
with  
   \begin{align*}
F_R(y):=\dfrac{1}{2}\iint_Q \rho_R^2 |y|^2\,dx\,dt,\quad G(v):=\dfrac{1}{2}\iint_{Q_\omega} \rho_0^2 |v|^2 dx\,dt\quad \text{and}\quad \overline{F}_R(z):=F_R(z+\overline{y}) .
   \end{align*}

   For every~$R>0$, we will consider the family of perturbations~$\{\Phi_R(\cdot\,,p)\}_{p\in L^2(Q)}$, defined as follows:
   \begin{align*}
\Phi_R(v,p):=\overline{F}_R(Mv-p)+G(v) \quad \forall v\in \mathcal{U},\ \ \forall p\in L^2(Q).
   \end{align*}

   We will adapt the strategy and results in~\cite{Ekeland1999Convex}.
   
   Clearly, $\Phi(v,0)=I_R(v)$ for all~$ v\in \mathcal{U}$. 
   Moreover, the convex conjugate~$\Phi_R^*$ of~$\Phi_R$ is given by 
   \begin{align*}
\Phi_R^*(w,q)=\overline{F}_R^*(-q)+G^*(M^*(w+q)) ,
   \end{align*}
where
   \begin{align*}
\overline{F}_R^*(u):=\dfrac{1}{2}\iint_Q \rho_R^{-2} |u|^2\,dx\,dt-\iint_Q \overline{y}u\,dx\,dt,\quad G^*(w):=\dfrac{1}{2}\iint_{Q_\omega} \rho_0^{-2} |w|^2 dx\,dt .
   \end{align*} 
   
   In particular, we see that
   \begin{align*}
	\Phi_R^*(0,q)=&\overline{F}_R^*(-q)+G^*(M^*q)\\
	=&\dfrac{1}{2}\iint_{Q} \rho_R^{-2}|q|^2\,dx\,dt +\dfrac{1}{2}\iint_{Q_\omega} \rho_0^{-2}|M^*q|^2\,dx\,dt + \iint_{Q} q\overline{y}\,dx\,dt
   \end{align*}
for every~$q\in L^2(Q)$.

   The dual problem of~\eqref{extremal:problem:02} is
   \begin{align}
	\label{extremal:problem:03}
	\begin{cases}
		\textrm{Maximize }-\Phi_R^*(0,q) , \\
		\textrm{Subject to }q\in L^2(Q),
	\end{cases}
   \end{align}
while the Lagrangian associated with~$\Phi_R$ reads
   \begin{align*}
	L_R(v,q) := &~-F_R^*(-q)+G(v)-(q,Mv+\overline{y})_{L^2(Q)}\\
	=&~\dfrac{1}{2}\iint_{Q_\omega} \rho_0^2 |v|^2\,dx\,dt -\dfrac{1}{2}\iint_{Q} \rho_R^{-2}|q|^2\,dx\,dt -\iint_{Q} q(Mv+\overline{y})\,dx\,dt.
   \end{align*}

\begin{theorem}\label{thm:existence:saddle:point:Lagrangian}
   Let~$R>0$ be given and let the functions~$\rho_R$ and~$\rho_0$ be as in~\eqref{weight:conditions}, \eqref{Carleman:weights} and~\eqref{weight:rho_R}.
   Then, one has:

\begin{itemize}

\item The extremal problems~\eqref{extremal:problem:02} and~\eqref{extremal:problem:03} are uniquely solvable.

\item The following identities hold:
	\begin{align*}
\inf_{v\in \mathcal{U}}I_R(v)=\inf_{v\in \mathcal{U}} \sup_{q\in L^2(Q)} L_R(v,q) \ \text{ and }\ \sup_{q\in L^2(Q)} -\Phi_R^*(0,q)=\sup_{q\in L^2(Q)} \inf_{v\in \mathcal{U}} L_R(v,q) .
    \end{align*}
	 
\item Let~$v_R\in \mathcal{U}$ and~$q_R\in L^2(Q)$ be the solutions respectively to~\eqref{extremal:problem:02} and~\eqref{extremal:problem:03}.
   Then~$(v_R,q_R)$ is the unique saddle-point of~$L_R$, that is, the unique couple in~$\mathcal{U}\times L^2(Q)$ satisfying 
         \begin{align*}
\inf_{v\in \mathcal{U}} \sup_{q\in L^2(Q)} L_R(v,q) = L_R(v_R,q_R)=\sup_{q\in L^2(Q)}\inf_{v\in \mathcal{U}} L_R(v,q).
	 \end{align*}
	
\item The following optimality characterization holds
         \begin{align}\label{extremality:conditions}
\begin{cases}
\overline{F}_R(Mv_R)+\overline{F}_R^*(-q_R)=-(Mv_R,q_R)_{L^2(Q)},\\
G(v_R)+G^*(M^*q_R)=\langle M^*q_R,v_R \rangle_{\mathcal{U}',\mathcal{U}}.
\end{cases}
	 \end{align}
   Consequently, setting~$y_R := Mv_R+\overline{y}$, one has 
	 \begin{align}\label{solutions:yR:vR}
y_R=-\rho_R^{-2}q_R,\quad v_R=\rho_0^{-2}M^*q_R.
	 \end{align}

\end{itemize}
\end{theorem}

\

   All these assertions are immediate consequences of the results in~\cite{Ekeland1999Convex} (see Chapter III, Proposition 4.2) except the fact that~\eqref{extremality:conditions} implies~\eqref{solutions:yR:vR}. But this follows from the values taken by~$\overline{F}_R$ and~$G$ and their convex conjugate functions.

   Indeed, for instance, the first identity in~\eqref{extremality:conditions} reads
   \begin{align*}
\dfrac{1}{2}\iint_Q \rho_R^2 |Mv_R+\overline{y}|^2\,dx\,dt+\dfrac{1}{2}\iint_{Q} \rho_R^{-2}|q_R|^2\,dx\,dt =-\iint_{Q} (Mv_R+\overline{y})q_R\,dx\,dt,
   \end{align*} 
that is to say,
   \begin{align*}
\iint_{Q} |\rho_R y_R+\rho_R^{-1} q_R|^2 \,dx\,dt=0
   \end{align*}
and the first equality in~\eqref{solutions:yR:vR} is found.
   A similar argument  from the second identity in~\eqref{extremality:conditions} leads to the second one.

\subsection{The Augmented Lagrangian}\label{Section:Augmented:Lagrangian:methods}

   Before going further, let us introduce the function
   \begin{align*}
	\mathcal{L}_R(y,v,q) := J_R(v,y)-(q,Mv+\overline{y}-y)_{L^2(Q)} \quad \forall (y,v,q) \in L^2(Q)\times \mathcal{U}\times L^2(Q).
   \end{align*}

   The unique saddle-point~$(v_R,q_R)$ of~$L_R$ is related to the unique saddle-point~$(v_R,y_R,q_R)$ of~$\mathcal{L}_R$ through the identity~$y_R=Mv_R+\overline{y}$.
   
   Indeed, we have by definition 
   \begin{align*}
	L_R(v,q)=\inf_{p\in L^2(Q)}\{\Phi_R(v,p)-(q,p)_{L^2(Q)}\}=\inf_{y\in L^2(Q)} \mathcal{L}_R(y,v,q),
   \end{align*}
where we have used the change of variables~$y=Mv+\overline{y}-p$.
   Then, it is clear that
   \begin{align*}
	\sup_{q\in L^2(Q)} \inf_{v\in \mathcal{U}} L_R(v,q)=\sup_{q\in L^2(Q)} \inf_{(y,v)\in L^2(Q)\times \mathcal{U}} \mathcal{L}_R(y,v,q)
   \end{align*} 
and also 
   \begin{align*}
\inf_{v\in \mathcal{U}} \sup_{q\in L^2(Q)} L_R(v,q)=\inf_{(y,v)\in L^2(Q)\times \mathcal{U}}\sup_{q\in L^2(Q)} \mathcal{L}_R (y,v,q).
   \end{align*}
   Consequently, for every~$R>0$, $\mathcal{L}_{R}$ possesses a unique saddle-point~$(v_R,y_R,q_R)\in L^2(Q)\times \mathcal{U}\times L^2(Q)$, with~$(y_R,v_R)$ being the unique solution to~\eqref{truncated:problem}.
   
   From now on, it will be said that~$y$ and~$v$ are the {\it primal} variables and~$q$ is the {\it dual} variable.
      
\

   In order to improve the convergence of Lagrangian methods, it is natural to consider the {\it Augmented Lagrangian} associated to~$J_R$.
   For any given~$R>0$ and~$K > 0$, it is given by
   \begin{align*}
	\begin{split} 
	\mathcal{L}_{R,K}(y,v,q):=\mathcal{L}_R(y,v,q)+\dfrac{K}{2}\|Mv+\overline{y}-y\|_{L^2(Q)}^2 \quad \forall (y,v,q)\in L^2(Q)\times \mathcal{U}\times L^2(Q).
	\end{split}
   \end{align*} 
   
%
%
%

\

   The next result asserts that the Lagrangians~$\mathcal{L}_{R}$ and~$\mathcal{L}_{R,K}$ have the same (unique) saddle-point:

\begin{theorem}
	Let~$R>0$ and~$K>0$ be given. The following holds:
	
\begin{itemize}
	
\item[(a)] Every saddle-point of~$\mathcal{L}_{R}$ is a saddle-point of~$\mathcal{L}_{R,K}$ and conversely.
   Consequently, there exists a unique saddle-point~$(y_R,v_R,q_R)$ of~$\mathcal{L}_{R,K}$ that satisfies~\eqref{solutions:yR:vR}.
		
\item[(b)] Furthermore, $\mathcal{L}_{R,K}(y_R,v_R,q_R) = \mathcal{L}_{R}(y_R,v_R,q_R)$.
	
\end{itemize}
\end{theorem}

\begin{proof}
   Suppose that~$(y,v,q)\in L^2(Q)\times \mathcal{U}\times L^2(Q)$ is a saddle point of~$\mathcal{L}_{R}$.
   Then,
   \begin{align}
		\label{saddle:point:L:R}
		\mathcal{L}_R(y,v,p)\leq \mathcal{L}_{R}(y,v,q)\leq \mathcal{L}_R(z,w,q) \quad \forall p \in L^2(\Omega), \ \ \forall (z,w)\in L^2(Q)\times \mathcal{U} .
   \end{align}
   From the first inequality, we deduce that~$(p-q,Mv+\overline{y}-y)_{L^2(Q)}\leq 0$ for all~$q\in L^2(Q)$, whence
   \begin{align}
		\label{saddle:point:restriction}
		y = Mv + \overline{y}.
   \end{align}
   Therefore, from the definition of~$\mathcal{L}_{R,K}$ and the first inequality in~\eqref{saddle:point:L:R}, it is clear that 
   \begin{align*}
		\mathcal{L}_{R,K}(v,y,p)\leq \mathcal{L}_R(v,y,q) \quad \forall p\in L^2(Q).
   \end{align*}
	
	On the other hand, from the second inequality of~\eqref{saddle:point:L:R}, using~\eqref{saddle:point:restriction}, we see that
   \begin{align*}
		\mathcal{L}_{R,K}(y,v,q)=\mathcal{L}_{R}(y,v,q)\leq \mathcal{L}_R(z,w,q)\leq \mathcal{L}_{R,K}(z,w,q) \quad \forall (z,w)\in L^2(Q)\times \mathcal{U} .
   \end{align*}
	This proves that~$(y,v,q)$ is a saddle point of~$\mathcal{L}_{R,K}$.
	
	Conversely, suppose that~$(y,v,q)$ is a saddle point of~$\mathcal{L}_{R,K}$. Then
   \begin{align}
		\label{saddle:point:L:R:kappa}
		\mathcal{L}_{R,K}(y,v,p)\leq \mathcal{L}_{R,K}(y,v,q)\leq \mathcal{L}_{R,K}(z,w,q)\quad \forall p \in L^2(\Omega), \ \ \forall (z,w)\in L^2(Q)\times \mathcal{U} .
   \end{align}
From the first inequality, we deduce again~\eqref{saddle:point:restriction}. In particular, this means that 
   \begin{align*}
		\mathcal{L}_R(y,v,p)\leq \mathcal{L}_R(y,v,q) \quad \forall q\in L^2(Q).
   \end{align*}
	 
	Now, the second inequality of~\eqref{saddle:point:L:R:kappa} implies that~$(y,v)$ is the unique minimizer of the strictly convex and differentiable function~$\mathcal{L}_{R,K}(\cdot\,,\cdot\,,q)$.
	The characterization of~$(y,v)$ as a minimizer together with~\eqref{saddle:point:restriction} implies the following identities: 
   \begin{align*}
\rho_R^{2}y +q=0,\quad v-\rho_0^{-2}M^*q=0.
   \end{align*}
	
	On the other hand, from the convexity of~$\mathcal{L}_R$, we deduce that, for each~$(z,w)\in L^2(Q)\times \mathcal{U}$, one has	
   \begin{align*}
		\mathcal{L}_R(z,w,q) &\geq \mathcal{L}_R(y,v,q)+(\nabla_{y,v} \mathcal{L}_R(y,v,q),(y-z,v-w))_{L^2(Q) \times L^2(Q_\omega)}\\
		&=\mathcal{L}_R(y,v,q).
   \end{align*}
	
	This proves that~$(y,v,q)$ is a saddle point of ~$\mathcal{L}_{R}$ and therefore assertion {\it(a)} holds.
	
	Assertion~{\it(b)} is deduced directly from~{\it(a)} and~Theorem~\ref{thm:existence:saddle:point:Lagrangian}.  
\end{proof}

\section{Some iterative algorithms}
\label{Section:iterative:algorithms}
\subsection{Uzawa's algorithms}

   In this section, we will indicate how to solve~\eqref{truncated:problem} using {\it Uzawa's algorithm.}
   Recall that this is just the optimal step gradient method applied to the dual problem~\eqref{extremal:problem:03}. 

   Thus, let us set 
   \begin{align*}
	\mathcal{J}_R^*(q):=\Phi_R^*(0,q)=\dfrac{1}{2}\iint_{Q} (\rho_R^{-2}|q|^2+\rho_0^{-2}|M^*q|^2 +2q\overline{y})\,dx\,dt \quad \forall q\in L^2(Q)
   \end{align*} 
and let us rewrite the dual problem~\eqref{extremal:problem:03} in the form 
   \begin{align}
	\label{problem:mathcalJ}
	\begin{cases}
		\textrm{Minimize }\mathcal{J}_R^*(q)\\
		\textrm{Subject to }q\in L^2(Q).
	\end{cases}
   \end{align}

%


   We note that~$\mathcal{J}_R^*$ is a quadratic functional.
   More precisely, one has
   \begin{align*}
	\mathcal{J}_R^*(q)=\dfrac{1}{2}a_R(q,q) + \ell (q) \quad \forall q \in L^2(Q),
   \end{align*}
where~$a_R(\cdot\,,\cdot)$ and~$\ell(\cdot)$ are respectively given by 
   \begin{align*}
	a_R(q,q'):=\iint_{Q} (\rho_R^{-2}qq'+\rho_0^{-2}M^*qM^*q')\,dx\,dt,
	\quad \ell(q'):=\iint_{Q} \overline{y} q' \,dx\,dt .
   \end{align*} 

   The bilinear form~$a_R(\cdot\,,\cdot)$ is symmetric and coercive on~$L^2(Q)\times L^2(Q)$.
   Consequently, it is completely natural and appropriate to apply Uzawa's algorithm, denoted ALG~1 in this paper.
   

\begin{algorithm}[h]
\caption{Uzawa's Algorithm (ALG~1, optimal step gradient version).}
\begin{algorithmic}[1]

\Procedure{Uzawa01}{$\overline{y},q^0,R,itmax,tol$}       
    \State Set~$err\leftarrow 1$ and~$k\leftarrow 0$.

    \While{$k \leq itmax$ and~$err\geq tol$ }  
    	\State Compute
    	$d^k\leftarrow \rho_R^{-2}q^k+M(\rho_0^{-2}M^*q^k)+\overline{y}.$ 
    	\State Compute the optimal step 
       \begin{align}
		\label{rk:optimal:step}
		&\displaystyle r_k\leftarrow \,\argmin\limits_{r\in \mathbb{R}_+}\,\mathcal{J}_R^*(q^k-rd^k),\\
		\nonumber 
		&q^{k+1}\leftarrow \,q^k-r_k d^k.
   \end{align}
	\If{$\displaystyle{\|q^{k+1}-q^{k}\|_{L^2(Q)}\over\|q^k\|_{L^2(Q)}}\leq tol$} \Comment{Convergence test}
		\State Set~$q\leftarrow q^{k+1}$.
		\Else 
		\State~$err \leftarrow \displaystyle{\|q^{k+1}-q^{k}\|_{L^2(Q)}\over\|q^k\|_{L^2(Q)}}$.
        \State~$k \leftarrow k+1$. 
	\EndIf
        
    \EndWhile
    \State  Define~$y\leftarrow -\rho_R^{-2}q$ and~$v\leftarrow \rho_0^{-2}M^*q$.
    \Comment{Computing the state and the control}
\EndProcedure

\end{algorithmic}
\end{algorithm}


   The weights~$\rho_R$ and~$\rho_0$ play a crucial role in the convergence properties of~ALG~1.
   Indeed, the Fr\'echet derivative of~$\mathcal{J}_R^*$ is given by
   \begin{align*}
	D\mathcal{J}_R^*(q)=\rho_R^{-2}q+M(\rho_0^{-2}M^*q)+\overline{y} \quad \forall q\in L^2(Q)
   \end{align*}
and it is therefore easy to see that
   \begin{align*}
	(D\mathcal{J}_R^*(q)-D\mathcal{J}_R^*(p),q-p)_{L^2(Q)}\geq C_1(R)\|q-p\|_{L^2(Q)}^2\quad \forall p,q\in L^2(Q),
   \end{align*}
with~$C_1(R):=\min\{\rho_R^{-2}(x,t):(x,t)\in Q\}
\thicksim R^{-2}$.

   On the other hand, we have 
   \begin{align*}
	\|D\mathcal{J}_R^*(q)-D\mathcal{J}_R^*(p)\|_{L^2(Q)}\leq C_0 \|q-p\|_{L^2(Q)}\quad \forall p,q\in L^2(Q),
   \end{align*}
where~$C_0 = (\|\rho_R^{-1}\|_\infty^2 +\|\rho_0^{-2}\|_\infty^2)^{1/2}$ does not depend of~$R>0$.
   Thus, if we are able to find~$a_0,b_0>0$ such that the optimal steps~$r_k$ defined in~\eqref{rk:optimal:step} satisfy 
   \begin{align*}
	a_0\leq r_k \leq b_0 < \dfrac{2C_1(R)}{C_0^2}\thicksim R^{-2} ,
   \end{align*}
then, from standard results concerning gradient algorithms, the convergence of~ALG~1 would be ensured.

    In view of the structure of the extremal problem~\eqref{problem:mathcalJ}, it is reasonable to introduce conjugate gradient iterates to improve the convergence of~ALG~1. 
    We will consider the so called {\it Polak-Ribi\`ere version}, here denoted ALG~2.

   Note that, as in the case of ALG~1, the one-dimensional extremal problems arising at each step are elementary.
   For ALG~2, the computational work is not much harder (just the computation of~$\gamma^k$ and the new~$d^k$) but well known results suggest a better (superlinear) convergence rate.


\begin{algorithm}[h]
	\caption{Uzawa's algorithm (ALG~2, Polak-Ribi\`ere conjugate gradient version).}
	\begin{algorithmic}[1]
		\Procedure{Uzawa02}{$\overline{y},q^0,R,itmax,tol$}
			\State Compute~$g^0\leftarrow \rho_R^{-2}q^0+M(\rho_0^{-2} M^*q^0)+\overline{y}.$
			\If{$\displaystyle{\|g^0\|_{L^2(Q)}\over\|q^0\|_{L^2(Q)}}\leq tol$}
				\State Set~$q\leftarrow q^0$ and~$err\leftarrow 0$.
			\Else 
			   \begin{align*}
					&d^0\leftarrow g^0,\\
					&r_0\leftarrow \,\argmin\limits_{r\in \mathbb{R}_+} \mathcal{J}_R^*(q^0-rd^0),\\
					&q^{1}\leftarrow \,q^0-r_0 d^0,\\
					& k\leftarrow 1.
			   \end{align*} 
				\State Set~
				$err\leftarrow 1$.
			\EndIf
			\While{$k\leq itmax$ and~$err\geq tol$}
				\State Compute 
				$g^k\leftarrow\rho_R^{-2} q^k + M(\rho_0^{-2} M^*q^k) + \overline{y}$.
				\Comment{Steepest descent}
				\State Compute \Comment{Construction of the new descent direction}
				   \begin{align*}
						&\gamma_k\leftarrow\,
						\dfrac{(g^{k},g^{k}-g^{k-1})_{L^2(Q)}}{\|g^{k-1}\|_{L^2(Q)}^2},							\quad d^k\leftarrow d^{k-1} + \gamma^k g^k
						\\
						&r_k\leftarrow \,\textrm{argmin}_{r\in \mathbb{R}_+} \mathcal{J}_R^*(q^k-							rd^k),\\
						&q^{k+1}\leftarrow \,q^{k} - r_k d^k.
				   \end{align*}

				\If{$\frac{\|g^{k}\|_{L^2(Q)}}{\|g^0\|_{L^2(Q)}}\leq tol$} \Comment{Converge test}
					\State Set~$q\leftarrow q^{k}$					
				\Else
				\State~$err \leftarrow \displaystyle{\frac{\|g^{k}\|_{L^2(Q)}}{\|g^0\|_{L^2(Q)}}}$.
				\State Set~$k\leftarrow \, k+1$.
				\EndIf 
			\EndWhile
			\State Set~$y\leftarrow -\rho_R^{-2}q$ and~$v\leftarrow \rho_0^{-2} M^* q$. 
		\EndProcedure
	\end{algorithmic}
\end{algorithm}

%

\subsection{Augmented Lagrangian methods}

   We can try to argue as before and produce algorithms similar to ALG~1 and ALG~2 for~$\mathcal{L}_{R,K}$, more precisely, for the extremal problem
   \begin{align}\label{aug-Uzawa}
	\begin{cases}
		\text{Minimize }\displaystyle \mathcal{J}^*_{R,K}(q) \\
		\text{Subject to } q \in L^2(Q),  
	\end{cases}
   \end{align}  
where
   \begin{align*}
\mathcal{J}^*_{R,K}(q) = - \min_{(y,v) \in L^2(Q) \times \mathcal{U}} \mathcal{L}_{R,K}(y,v,q)\quad \forall q \in L^2(Q).
   \end{align*}
   
   Let us simplify the notation and denote~$\mathcal{J}^*_{R,K}$ by~$\mathcal{B}$.
   Then, it can be checked after some computations that, for every~$q \in L^2(Q)$, one has
   \begin{align*}
\mathcal{B}'(q) = Mv(q) + \overline{y} - y(q),
   \end{align*}  
where~$(y(q),v(q))$ is the unique solution to the linear system
   \begin{align}\label{3.7p}
   \left\{
	\begin{array}{cccl} \ds
		(K\rho_R^{-2}+1) y & - & K\rho_R^{-2} Mv &= -\rho_R^{-2}(q-K \overline{y}) ,
		\\ \noalign{\medskip} \ds
		-K \rho_0^{-2}M^*y & + & (K\rho_0^{-2} M^*M + \mathbbm{1}_\omega)v &= 
		\rho_0^{-2}M^*(q-K \overline{y}) .
	\end{array}
   \right.
   \end{align}  
   
   Note that, for~$K = 0$, \eqref{3.7p} is just~\eqref{solutions:yR:vR} for~$q = q_R$.
   
   The resolution of~\eqref{3.7p} can be easily performed using the auxiliary variable~$m:=Mv+\overline{y}-y$.
   
   Indeed, observe that for any~$K > 0$~\eqref{3.7p} can be rewritten as follows:
   \begin{align}\label{3.7pp}
   \left\{
\begin{array}{cccl} \ds
		\dfrac{1}{K} \rho_0^2 v  & + & M^*m & = \dfrac{1}{K} M^*q ,
		\\ \noalign{\medskip} \ds
		-Mv & + & (1+K\rho_R^{-2}) m & = \overline{y} + \rho_R^{-2}q ,
\end{array}
   \right.
   \end{align}
with~$y = Mv + \overline{y} - m$.
   The analog of Uzawa's algorithm ALG~1, that is, the optimal step gradient algorithm for~\eqref{aug-Uzawa} will be denoted ALG~3.

\begin{algorithm}[!h]
	\caption{Uzawa's algorithm (ALG~3, optimal step gradient for Augmented Lagrangian).}
	\begin{algorithmic}[1]
		\Procedure{Uzawa03}{$\overline{y},q^0,R,K,itmax,tol$}
		\State Set~$err \leftarrow 1$ and~$k\leftarrow 0$.
		\State Define~$v^0\leftarrow \rho_0^{-2}M^*q^0$ and~$m^0\leftarrow \rho_R^{-2}q^0+Mv^0+\overline{y}$.
		\While{$k\leq itmax$ and~$err\geq tol$}
			\State Compute~$v^k$ and~$m^k$ solving the system
			\begin{align*}
  			 \left\{
			\begin{array}{cccl} \ds
			\dfrac{1}{K} \rho_0^2 v  & + & M^*m & = \dfrac{1}{K} M^*q^k ,
			\\ \noalign{\medskip} \ds
			-Mv & + & (1+K\rho_R^{-2}) m & = \overline{y} + \rho_R^{-2}q^k ,
		\end{array}
   		\right.
   		\end{align*}
			\State Compute the optimal step 	and update
			\begin{align*}
				r_k \leftarrow \text{argmin}_{r\in \mathbb{R}^+} \mathcal{J}_{R,K}^*(q^k-r m^k).
			\end{align*}
			\State Compute~$q^{k+1}$ as follows:
			\begin{align*}
				q^{k+1}\leftarrow q^k-r_k m^k.
			\end{align*}			
			\If{$\|q^{k+1}-q^k\|_{L^2(Q)}/\|q^k\|_{L^2(Q)}\leq tol$}
			\State Set~$q\leftarrow q^{k+1}$.
			\Else
			\State Compute~$err \leftarrow \|q^{k+1}-q^k\|_{L^2(Q)}/\|q^k\|_{L^2(Q)}$.
			\State Set~$k\leftarrow k+1$.
			\EndIf 			 
		\EndWhile
		\State Define~$y\leftarrow -\rho_R^{-2}q$ and~$v\leftarrow \rho_0^{-2}M^*q$.
		\EndProcedure 
	\end{algorithmic}
\end{algorithm} 

   Note that the linear system~\eqref{3.7pp} cannot be solved directly. For this reason, we shall perform Gauss-Seidel iterates to tackle this problem.
   
   In practice, in order to apply any of the previous algorithms, we must be able to compute (numerical approximations of) the~$Mv$ and the~$M^*q$ for various~$v$ and~$q$.
   
   In the numerical experiments in~Section~\ref{Section:numerical:experiments}, this is achieved through a standard finite dimensional reduction process that consists of
      
\begin{enumerate}

\item Implicit Euler or Gear time discretization.

\item Finite element approximation in space of the resulting Poisson-like problems.

\end{enumerate}


\section{Numerical null controllability of Stokes systems}
\label{Section:Stokes}

   The results in the previous sections can be adapted to the solution of the null controllability problem for the Stokes system. 
   
   Thus, let us introduce the spaces
   \[
\begin{array}{c}\ds
	\bm{H}:=\{\bm{\varphi}\in \bm{L}^2(\Omega)\,:\, \nabla \cdot \bm{\varphi} = 0 \text{ in }\Omega,\  \bm{\varphi}\cdot \bm{\nu} =0\text{ on }\partial \Omega \},
	\\ \noalign{\medskip} \ds
	\bm{V}:=\{\bm{\varphi}\in \bm{H}_0^1(\Omega)\,:\, \nabla \cdot \bm{\varphi} = 0 \text{ in }\Omega \},
	\\ \noalign{\smallskip} \ds
	U:=\{\psi\in H^1(\Omega)\,:\, \int_\Omega \psi(\bm{x})\, d\bm{x}=0\}.
\end{array}
   \]
%
%
%

   Let us fix~$\bm{y}^0\in \bm{H}$ and~$T>0$.
   For every $\bm{v}\in \bm{L}^2(Q_\omega)$, there exists exactly one solution~$(\bm{y},\pi)$ to
   \begin{align}
	\label{Stokes:system:01}
	\begin{cases}
		\pt \bm{y}-a \Delta \bm{y}+\nabla \pi=\mathbbm{1}_\omega \bm{v}&\text{ in }Q,\\
        \nabla \cdot \bm{y}=0&\text{ in }Q,\\
        \bm{y}=\bm{0}&\text{ on }\Sigma,\\
        \bm{y}(\cdot\,,0)=\bm{y}^0&\text{ in }\Omega,
    \end{cases}
   \end{align} 
with
   \begin{align*}
	\bm{y}\in C^0([0,T];\bm{H})\cap L^2(0,T;\bm{V}),\quad \pi \in L^2(0,T;\bm{U}).
   \end{align*}    

   The null controllability problem for the Stokes system~\eqref{Stokes:system:01} can be formulated as follows: for every~$\bm{y}^0\in \bm{H}$, find a control~$\bm{v}\in \bm{L}^2(Q_\omega)$ such that the associated solution to~\eqref{Stokes:system:01} satisfies 
   \begin{align*}
	\bm{y}(\bm{x},T)=\bm{0},\quad \bm{x}\in \Omega. 
   \end{align*}

   It is known that this problem is solvable, see~\cite{Fursikov1996Controllability}.
   Moreover, with the notation introduced in~Sections~\ref{Sec-1} and~\ref{Section:study:extremal:problem:trunctation}, we can consider the constrained extremal problem 
   \begin{align}
	\label{Opti:problem:Stokes}
	\begin{cases}
		\displaystyle \text{Minimize }J(\bm{y},\bm{v}):=\dfrac{1}{2}\iint_{Q}\rho_R^2 |\bm{y}|^2\,d\bm{x}\,dt + \dfrac{1}{2}\iint_{Q_\omega} \rho_0^2 |\bm{v}|^2\,d\bm{x}\,dt,\\
		\text{Subject to }(\bm{y},\bm{v})\in \mathcal{S}(\bm{y}^0,T),
	\end{cases}
   \end{align}
where
   \begin{align*}
	\mathcal{S}(\bm{y}^0,T):=\{(\bm{y},\bm{v})\in \bm{L}^2(Q) \times \mathcal{U}\,:\, \bm{y}\text{ solves }\eqref{Stokes:system:01}\text{ together with some }\pi\in L^2(0,T;\bm{U}) \}
   \end{align*}
and
   \begin{align*}
	\mathcal{U}:=\{\bm{v}\in \bm{L}^2(Q_\omega)\,:\, \rho_0v\in \bm{L}^2(Q_\omega)\}.
   \end{align*}

   Let~$(\bar{\bm{y}},\bar{\pi})$ be the solution to~\eqref{Stokes:system:01} with~$\bm{v}=\bm{0}$ and let~$\bm{M}:\bm{L}^2(Q_\omega) \to \bm{L}^2(0,T;\bm{H})$ be the linear operator that assigns to~$\bm{w}$ the velocity field~$\bm{u}$, where~$(\bm{u},\zeta)$ is the solution to the system
\begin{align*}
	\begin{cases}
		\pt \bm{u}-a \Delta \bm{u}+\nabla \zeta=\mathbbm{1}_\omega \bm{w}&\text{ in }Q,\\
        \nabla \cdot \bm{u}=0&\text{ in }Q,\\
        \bm{u}=\bm{0}&\text{ on }\Sigma,\\
        \bm{u}(\cdot\,,0)=\bm{0}&\text{ in }\Omega.
    \end{cases}
\end{align*}

   Note that the adjoint~$\bm{M}^*:\bm{L}^2(0,T;\bm{H})\to \bm{L}^2(Q_\omega)$ assigns to each~$\bm{\psi}$ the function~$\bm{\varphi}\big|_{Q_\omega}$, where the pair~$(\bm{\varphi},\sigma)$ satisfies
\begin{align*}
	\begin{cases}
		-\pt \bm{\varphi}-a\Delta \bm{\varphi}+\nabla \sigma =\bm{\psi}&\text{ in }Q,\\
		\nabla \cdot \bm{\varphi}=0&\text{ in }Q,\\
		\bm{\varphi}=\bm{0}&\text{ on }\Sigma,\\
		\bm{\varphi}(\cdot\,,T)=0&\text{ in }\Omega.
	\end{cases}
\end{align*}

%
%

   After some work, we see that the Lagrangian corresponding to~\eqref{Opti:problem:Stokes} is given by 
   \begin{align*}
\mathcal{L}_{R}(\bm{y},\bm{v},\bm{q}):=\dfrac{1}{2}\iint_{Q} \rho_R^2 |\bm{y}|^2\,d\bm{x}\,dt + \dfrac{1}{2}\iint_{Q_\omega} \rho_0^2 |\bm{v}|^2\,d\bm{x}\,dt - (\bm{q},\bm{M}\bm{v}+\bar{\bm{y}}-\bm{y})_{\bm{L}^2(Q)}
   \end{align*}
and, for any~$K > 0$, the Augmented Lagrangian is as follows:
   \begin{align*}
\mathcal{L}_{R,K}(\bm{y},\bm{v},\bm{q}):=\mathcal{L}_{R}(\bm{y},\bm{v},\bm{q}) + \dfrac{1}{2}\|\bm{M}\bm{v}+\bar{\bm{y}}-\bm{y}\|_{\bm{L}^2(Q)}^2 .
   \end{align*}

   Then, we can consider the extremal problem 
   \begin{align}\label{Stokes-AL-dual}
	\begin{cases}
		\text{Minimize }\mathcal{J}_{R,K}^*(\bm{q}),\\
		\text{Subject to }\bm{q}\in \bm{L}^2(Q),
	\end{cases}
   \end{align}
where
   \begin{align*}
	\mathcal{J}_{R,K}^*(\bm{q}):=-\min_{(\bm{y},\bm{v})\in \bm{L}^2(Q)\times \mathcal{U}} \mathcal{L}_{R,K}(\bm{y},\bm{v},\bm{q})\quad \forall \bm{q}\in \bm{L}^2(Q).
   \end{align*}
   
   The optimal step gradient method for~\eqref{Stokes-AL-dual}, that is, Uzawa's algorithm for the Augmented Lagrangian formulation, is denoted by~ALG~4.
   It is described below.
   
   Again, we must in practice be able to compute numerical approximations of the~$\bm{M}\bm{v}$ and the~$\bm{M}^*\bm{q}$.
   As in the case of the heat equation, this can be done in two steps:
   
\begin{enumerate}

\item By introducing Euler and Gear time discretization schemes that reduce the task to the solution of a finite family of stationary Stokes problems.

\item With a mixed finite element approximation of these Stokes systems.

\end{enumerate}



\begin{algorithm}[!h]
	\caption{Uzawa's algorithm (ALG~4, optimal step gradient for Stokes Augmented Lagrangian.}
	\label{ALG4}
	\begin{algorithmic}[1]
		\Procedure{Uzawa03}{$\overline{\bm{y}},\bm{q}^0,R,K,itmax,tol$}
		\State Set~$err \leftarrow 1$ and~$k\leftarrow 0$.
		\State Define~$\bm{v}^0\leftarrow \rho_0^{-2}M^*\bm{q}^0$ and~$\bm{m}^0\leftarrow \rho_R^{-2}\bm{q}^0+M\bm{v}^0+\overline{\bm{y}}$.
		\While{$k\leq itmax$ and~$err\geq tol$}
			\State Compute~$\bm{v}^k$ and~$\bm{m}^k$ solving the system
			\begin{align*}
  			 \left\{
			\begin{array}{cccl} \ds
			\dfrac{1}{K} \rho_0^2 \bm{v}  & + & M^*\bm{m} & = \dfrac{1}{K} M^*\bm{q}^k ,
			\\ \noalign{\medskip} \ds
			-M\bm{v} & + & (1+K\rho_R^{-2}) \bm{m} & = \overline{\bm{y}} + \rho_R^{-2}\bm{q}^k ,
		\end{array}
   		\right.
   		\end{align*}
			\State Compute the optimal step 	and update
			\begin{align*}
				r_k \leftarrow \text{argmin}_{r\in \mathbb{R}^+} \mathcal{J}_{R,K}^*(\bm{q}^k-r\bm{m}^k).
			\end{align*}
			\State Compute~$\bm{q}^{k+1}$ as follows:
			\begin{align*}
				\bm{q}^{k+1}\leftarrow \bm{q}^k-r_k\bm{m}^k.
			\end{align*}			
			\If{$\|\bm{q}^{k+1}-\bm{q}^k\|_{\bm{L}^2(Q)}/\|\bm{q}^k\|_{\bm{L}^2(Q)}\leq tol$}
			\State Set~$\bm{q}\leftarrow \bm{q}^{k+1}$.
			\Else
			\State Compute~$err \leftarrow \|\bm{q}^{k+1}-\bm{q}^k\|_{\bm{L}^2(Q)}/\|\bm{q}^k\|_{L^2(Q)}$.
			\State Set~$k\leftarrow k+1$.
			\EndIf 			 
		\EndWhile
		\State Define~$\bm{y}\leftarrow -\rho_R^{-2}\bm{q}$ and~$\bm{v}\leftarrow \rho_0^{-2}M^*\bm{q}$.
		\EndProcedure 
	\end{algorithmic}
\end{algorithm} 

\section{Numerical experiments}
\label{Section:numerical:experiments}

   In the sequel, we present the results of some numerical experiments in~$2$D in a rectangular domain and~$3$D in a cube. The computations have been performed with the \textit{FreeFem++} package (see \cite{Hecht2012New}). Moreover, the visualization of the results has been generated using appropriate MATLAB tools.
   
   In these experiments, we will focus on the behavior of the computed control and state. Besides, to validate the theoretical results of Propositions \ref{Proposition:convergence} and \ref{prop:estimate:y:T}, we also explore the behavior of the algorithm for various $R>0$. 

   First, let us suppose that~$\Omega=(0,1)\times (0,1)$ and~$T=0.5$ and let us consider the following system for the heat equation
\begin{align}
	\label{numerical:experiences:eq:01}
    \begin{cases}
        \pt y -0.1\Delta y -2.5 y=\mathbbm{1}_\omega v &\text{ in }Q,\\
        y=0&\text{ on }\Sigma,\\
        y(x_1,x_2,0)=\sin(\pi x_1)\sin(\pi x_2)&\text{ in }\Omega.
    \end{cases}
\end{align}

   Note that, in~\eqref{numerical:experiences:eq:01}, we have considered constant coefficients and a zero-order term with a negative sign.
   However, as indicated at the beginning of the paper, the ideas that follow can also be extended to non-constant coefficients, see for instance~\cite{FC2013Strong}.

   Let us analyze the uncontrolled solution to the problem.
   The norm~$\|y(\cdot\,,t)\|$ is increasing in time and, accordingly, the uncontrolled solution will not vanish at time~$t = T$.
   In~Figure~\ref{figure:projected:uncontrolled:state}, we depict the projections of the solution at~$x_1=0.35$ and~$x_2=0.4$.

   \begin{figure}[!h]
	\centering
	\begin{minipage}[b]{0.4\textwidth}
            \includegraphics[scale=0.332]{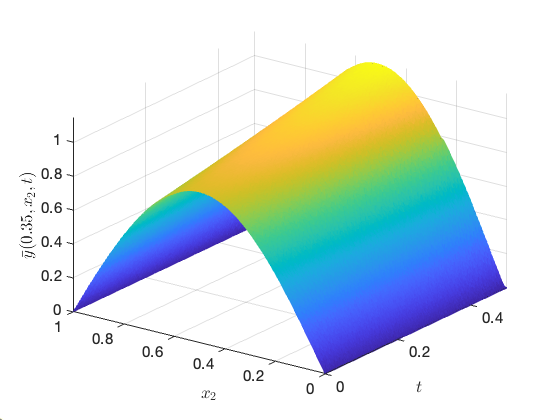}
    \end{minipage}
	\begin{minipage}[b]{0.4\textwidth}
            \includegraphics[scale=0.16]{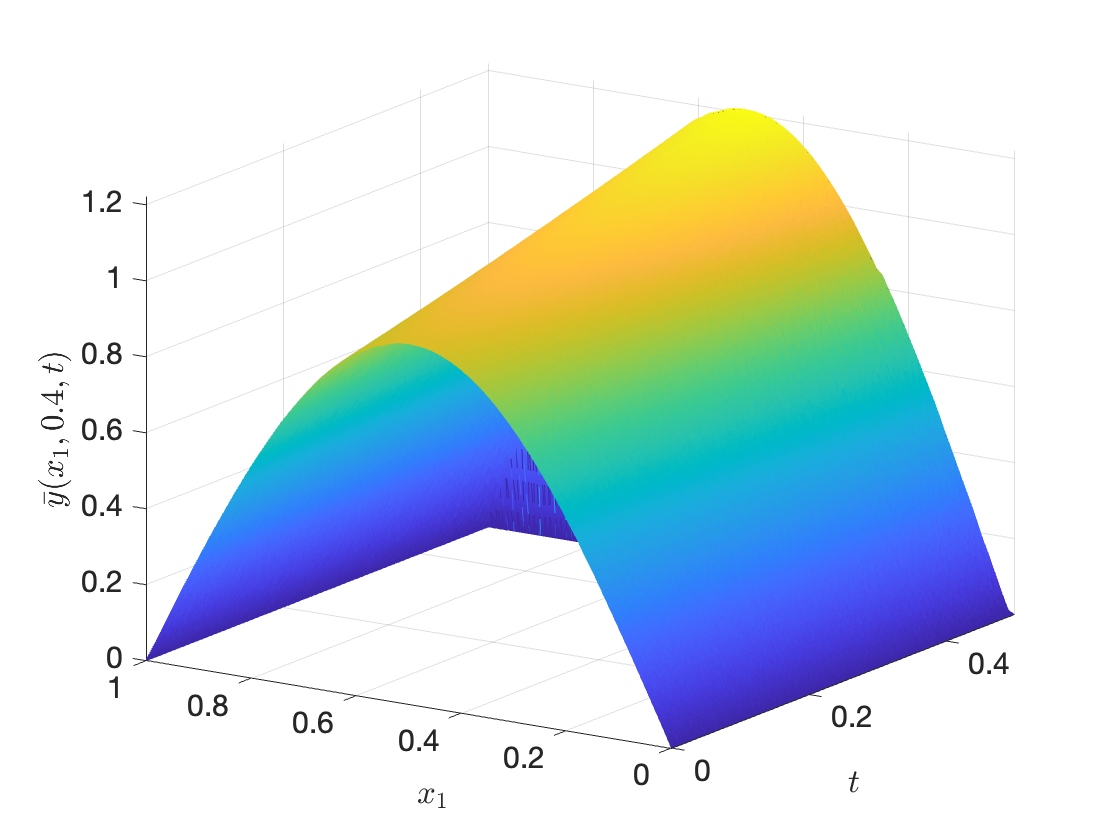}
    \end{minipage}
	\caption{Projected uncontrolled state at~$x_1=0.35$ and~$x_2=0.4$.}
	\label{figure:projected:uncontrolled:state}
   \end{figure}

   In order to apply our results, let us define the appropariate weight functions. Suppose that the control region can be written in the following way:
  ~$\omega = (a_1,b_1) \times (a_2,b_2)$, with~$a_i, b_i \in (0,1)$ and~$a_i < b_i$, for $i=1,2$.
   For~$s_i\in (0,1)$ with~$i=1,2$, we consider the following real-valued function~$\beta_{0,s_i}$:
   \begin{align*}
	\beta_{0,s_i}(x)=\dfrac{x(1-x)e^{-(x-c_{s_i})^2}}{s_i(1-s_i) e^{-(s_i-c_{s_i})^2}},\quad c_{s_i}:=s_i-\dfrac{1-2s_i}{2s_i(1-s_i)},\quad i=1,2.
   \end{align*}

   Then, we set
   \begin{align}
\label{numerical:eta:0}
\eta^0(x_1,x_2)=\beta_{0,s_1}(x_1)\beta_{0,s_2}(x_2) \quad \forall\, (x_1,x_2)\in \overline{\Omega}.
   \end{align}

   It is easy to check that~$\eta^0$ fulfills the Fursikov-Imanuvilov conditions in~\eqref{weight:conditions}.
   In the sequel, we define the weight functions~$\rho$ and~$\rho_0$ as in~\eqref{Carleman:weights}, with~$\eta^0$ given by~\eqref{numerical:eta:0}, $s_i=(b_i-a_i)/2$ for~$i=1,2$, $K_1=1/10$ and~$K_2=2\|\eta^0\|_{L^\infty}=2$. 

   Several numerical experiments concerning~\eqref{numerical:experiences:eq:01} using ALG~3 are presented in the following sections.
   In all of them, ALG~3 is initiated with~$q^0\equiv 0$ and the number of steps used for time discretization is denoted by~$N_t$.


\subsection{Test \#1: Controlling with~$\omega$ small enough}

   In this experiment, we consider the controlled heat equation where the action of the control is given in a small subset of the domain.
   
   More precisely, in this case the control function in~\eqref{numerical:experiences:eq:01} is supported in the set~$\omega=(0.2,0.5)\times (0.2,0.6)\subset \Omega$.
   Then, we use Uzawa's algorithm proposed in ALG~3 with~$R:=10^{5}$, $tol:=2\cdot 10^{-5}$, and~$K=0.1$.
   We use an initial mesh with 2358 triangles and take $N_t=80$.
   We adapt the mesh every 10 iterates according to the values of
   \begin{align*}
	\tilde{y}^{k}(x) :=
	\frac{1}{\Delta t} \sum_{n=0}^{N_t} y^{k}(x,t_n),\quad x\in \Omega.
   \end{align*}
where $t_n=n\Delta t$. In order to illustrate the evolution of the controlled solution, we have depicted the state~$y = y(x_1,x_2,t)$ for some selected times in~Figures~\ref{figure:evolution:controlled:state:01} and ~\ref{figure:evolution:controlled:state:02}.
   It is shown there that, as time evolves, the action of the control makes the computed solution locally negative;
   this is coherent with the {\it Maximum Principle} for the parabolic equation.
   Then, after a while, the solution tends rapidly to zero.

\begin{figure}[!h]
	\begin{minipage}[b]{0.5\textwidth}
	\centering
            \includegraphics[scale=0.30]{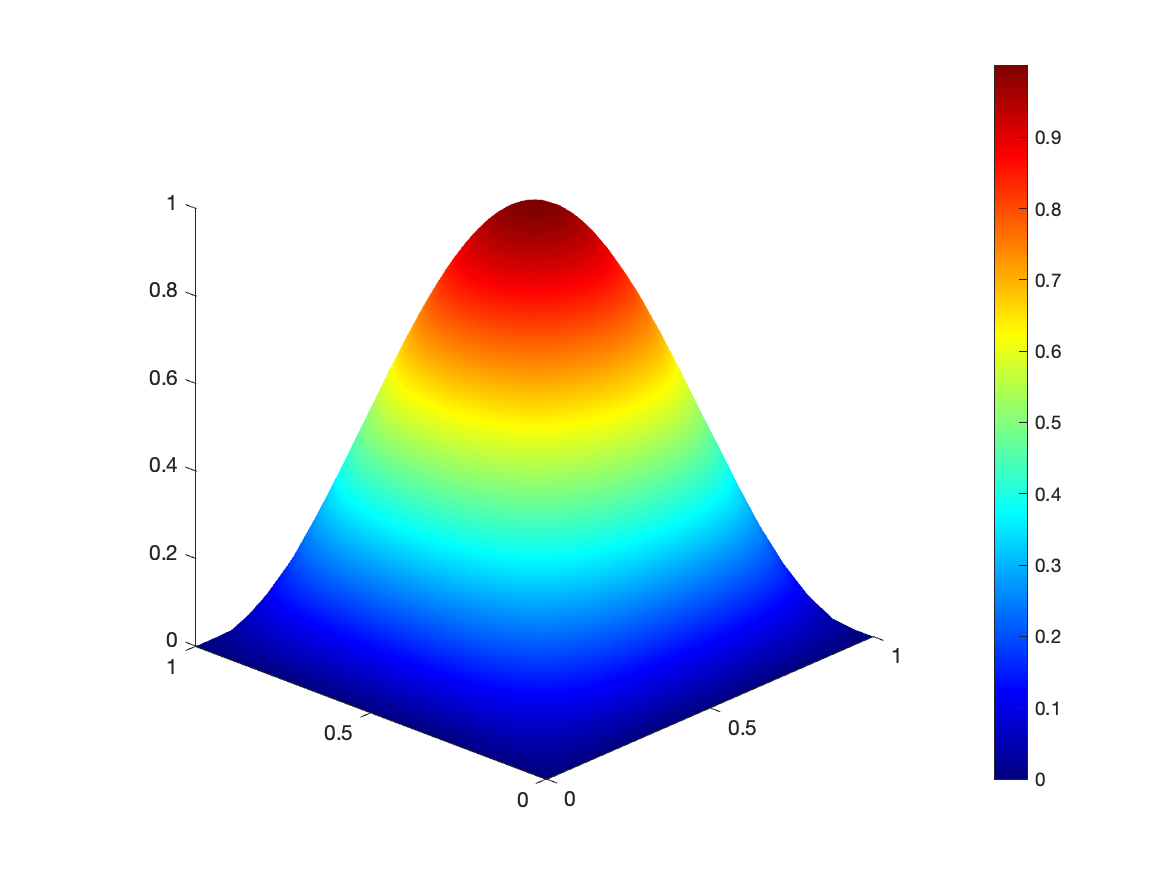}
    \end{minipage}
	\hfill 
    \begin{minipage}[b]{0.50\textwidth}
    	\centering
        \includegraphics[scale=0.30]{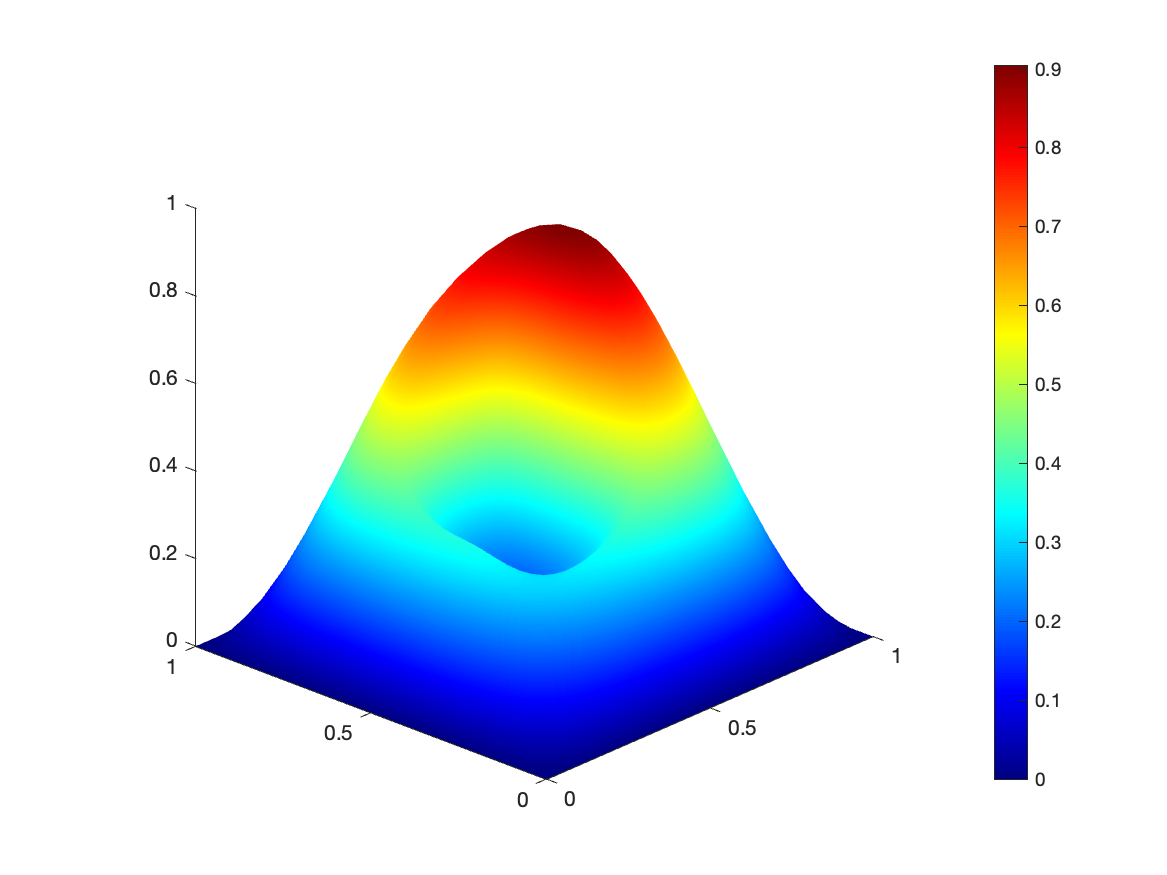}
    \end{minipage}
\vspace{1em}
        \begin{minipage}[b]{0.50\textwidth}
        	\centering
            \includegraphics[scale=0.30]{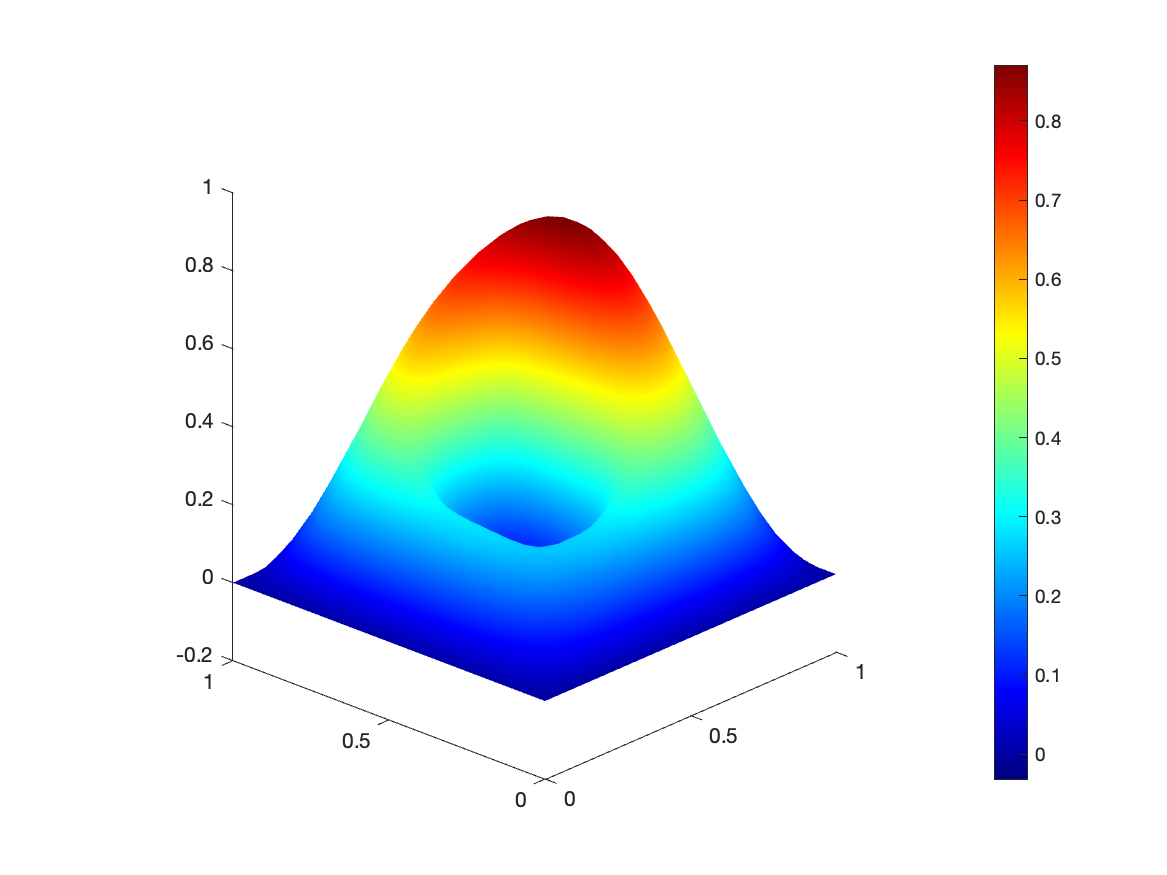}
        \end{minipage}
        \hfill
        \begin{minipage}[b]{0.50\textwidth}
        	\centering
            \includegraphics[scale=0.30]{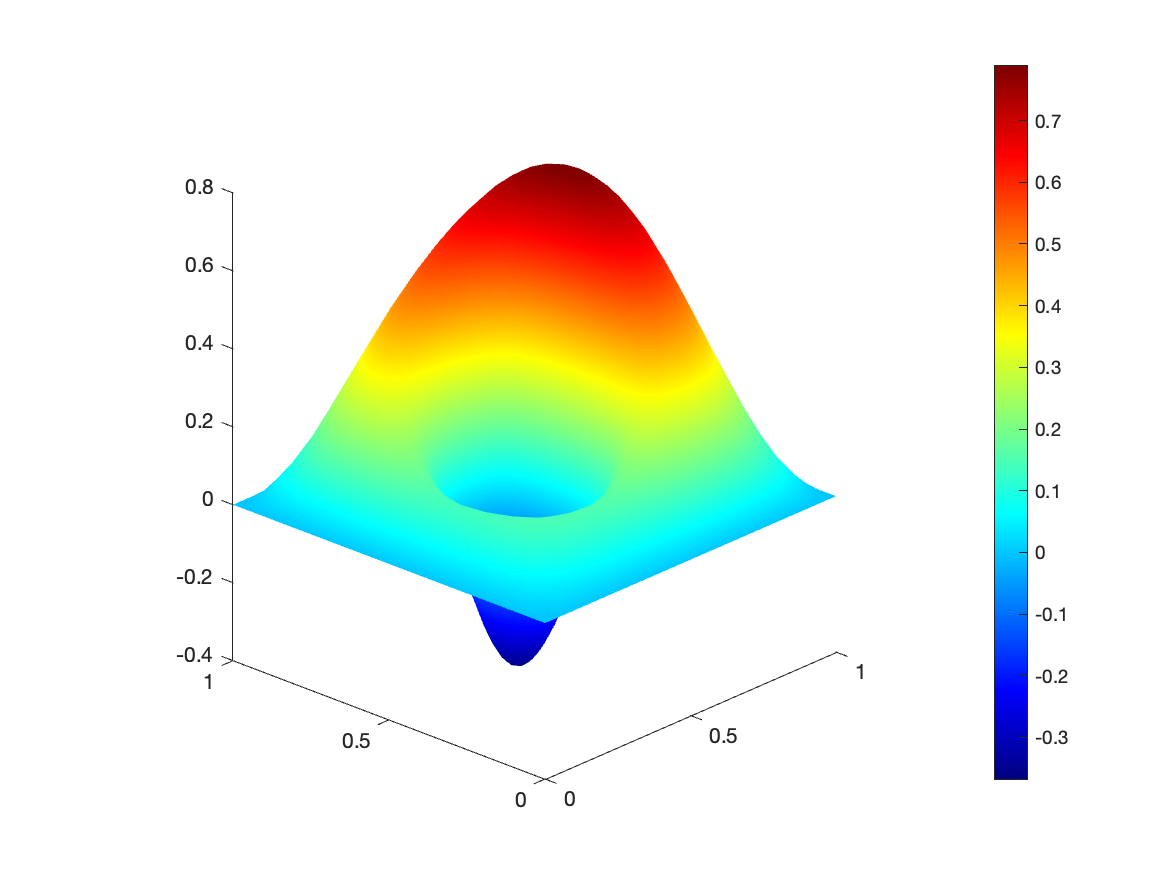}
        \end{minipage}
        \vskip-.5cm
        \caption{Evolution of the controlled state at~$t=0,\, 0.0625,\, 0.0875$ and~$0.15$ (From left to right).}
        \label{figure:evolution:controlled:state:01}
\end{figure}

\begin{figure}[!h]
	\begin{minipage}[b]{0.50\textwidth}
            \centering 
            \includegraphics[scale=0.30]{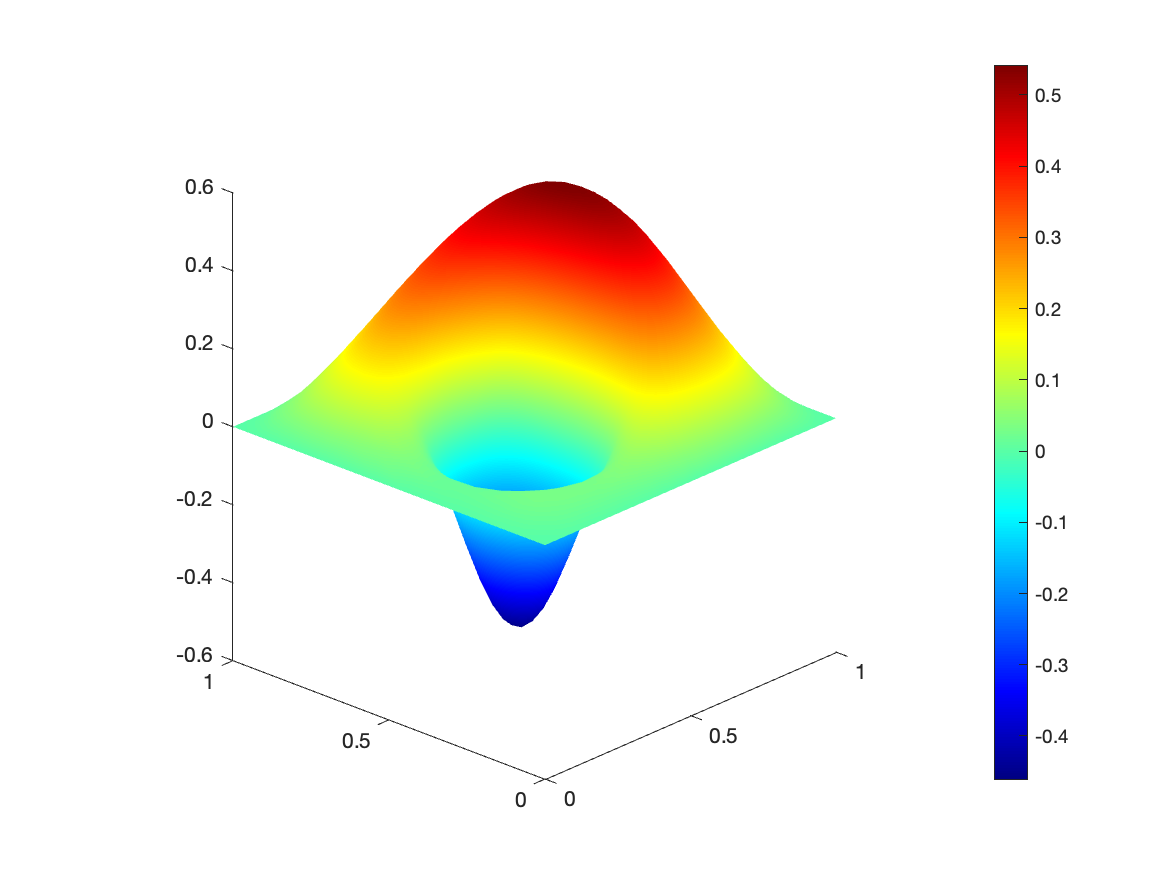}
    \end{minipage}
    \hfill
    \hfill
    \begin{minipage}[b]{0.50\textwidth}
        \centering 
        \includegraphics[scale=0.30]{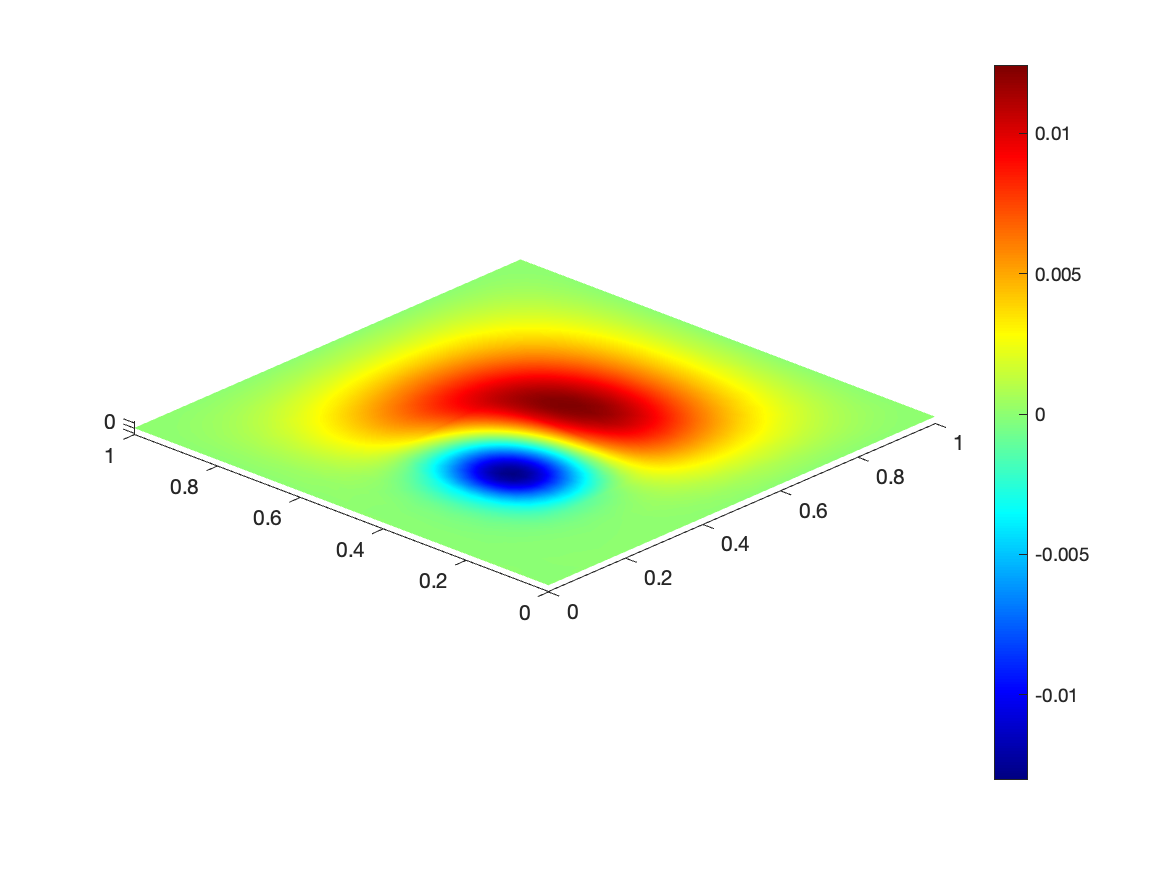}
    \end{minipage}
\vspace{1em}
        \begin{minipage}[b]{0.50\textwidth}
            \centering 
            \includegraphics[scale=0.30]{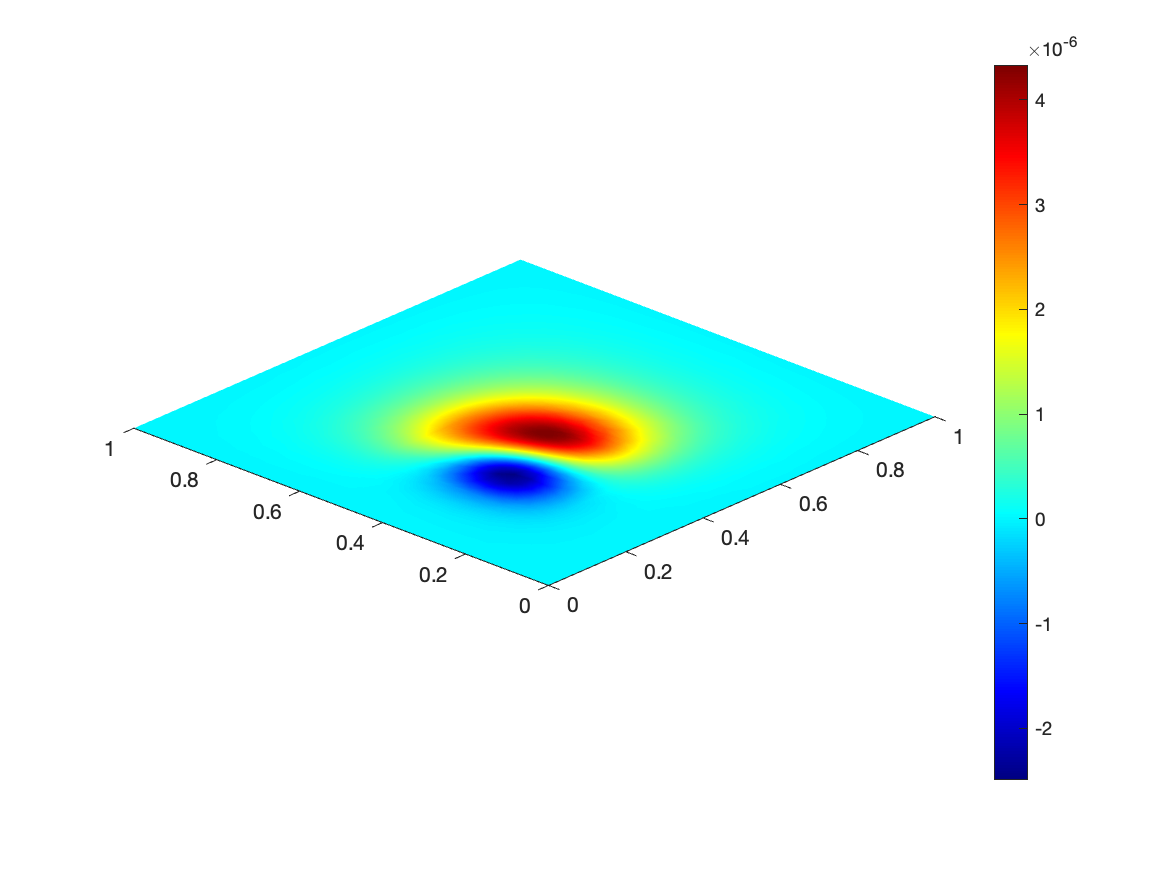}
        \end{minipage}
        \hfill
        \begin{minipage}[b]{0.50\textwidth}
            \centering 
            \includegraphics[scale=0.30]{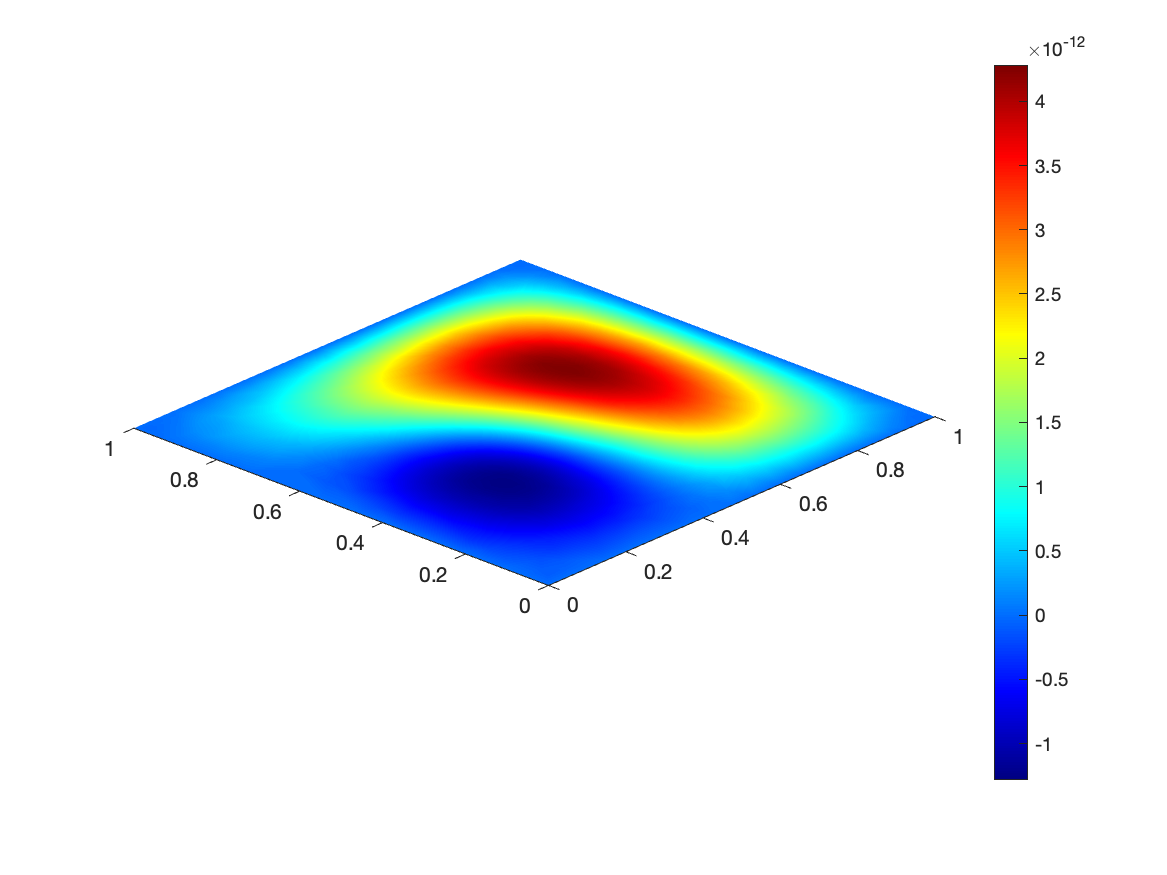}
        \end{minipage}
        \vskip-.5cm
        \caption{Evolution of the controlled state at~$t = 0.25,\, 0.375,\, 0.4375$ and~$0.5$ (from left to right).}
        \label{figure:evolution:controlled:state:02}
\end{figure}

On the other hand, the projections of the control computed by ALG~3 at~$x_1 = 0.35$ and then at~$x_2 = 0.3$ are exhibited in~Figure~\ref{figure:projected:control}.
   We observe that, in~Figure~\ref{figure:projected:control}, the control takes positive values for some values of~$t$ close to the final time.
   Again, this is coherent with the {\it Maximum Principle} for the heat equation.

   \begin{figure}[!h]
	\centering
	\begin{minipage}[b]{0.40\textwidth}
            \includegraphics[scale=0.19]{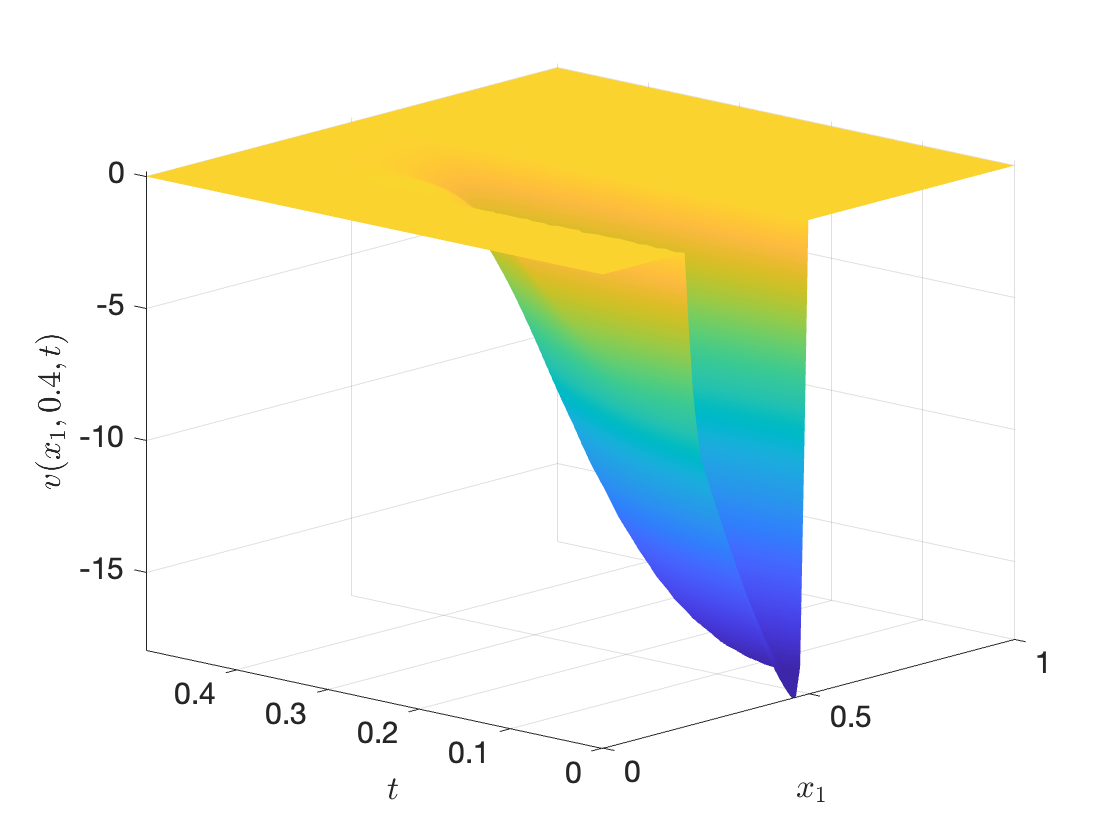}
    \end{minipage}
	\begin{minipage}[b]{0.40\textwidth}
            \includegraphics[scale=0.19]{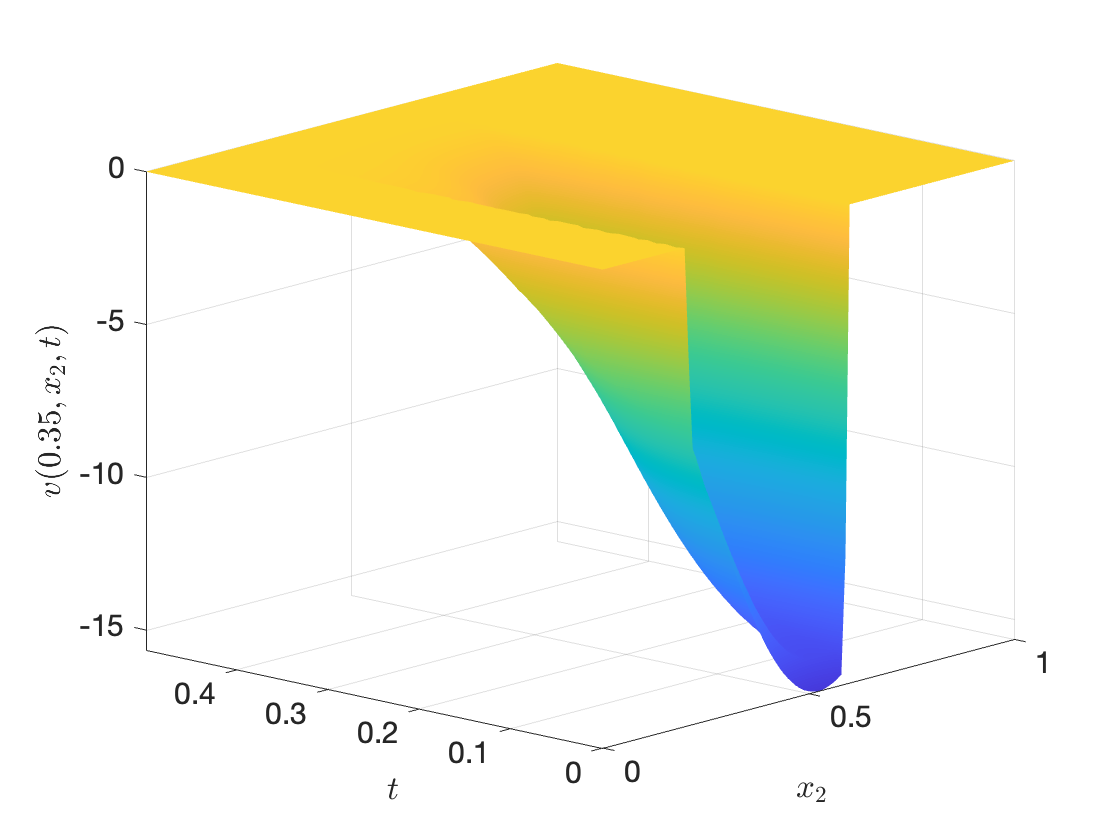}
    \end{minipage}
    \vskip-.5cm
	\caption{Projected uncontrolled state at~$x_1=0.35$ (left) and~$x_2=0.4$ (right).}
	\label{figure:projected:control}
  \end{figure}


   The norms of~$y$ and~$q$ at final time~$T$ are given by 
\begin{align*}
	\|y(\cdot\,,T)\|_{L^2(\Omega)}=1.71997\cdot 10^{-12},\quad \|q(\cdot\,,T)\|_{L^2(\Omega)}=35.0714.
\end{align*}
   For the evolution of the spatial norms of the uncontrolled and controlled states, the control and the computed~$q$, see~Figures~\ref{Figure:norm:ybar:y} and~\ref{Figure:norm:computed_q}.
   Note that the norm of the controlled state decays exponentially as~$t\to T^-$.
   
   The initial and final meshes obtained by the adapt mesh procedure used for the computations can be seen in~Figure~\ref{Figure:Initial:final:mesh}. Observe that they contain~$2358$ and~$1899$ triangles, respectively.

\begin{figure}[!h]
	\begin{minipage}[b]{0.5\textwidth}
		\centering 
		\includegraphics[scale=0.22]{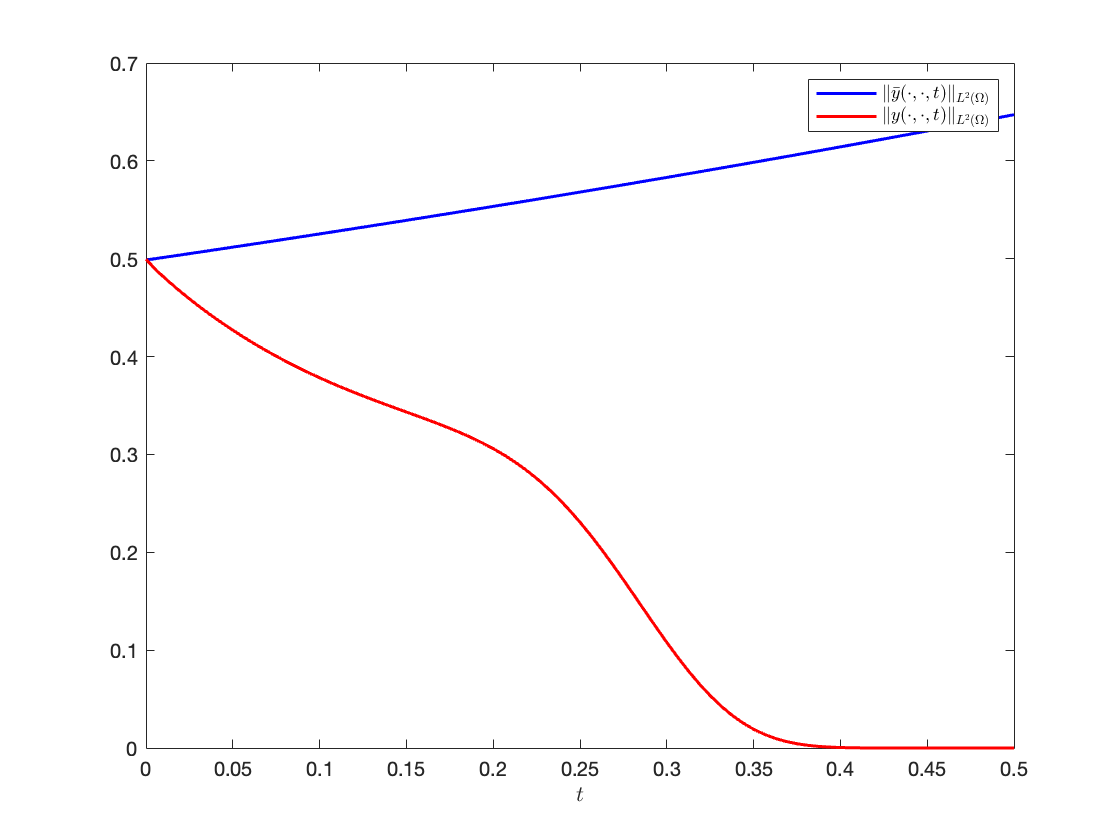}
	\end{minipage}
	\hfill 
	\begin{minipage}[b]{0.5\textwidth}
		\centering 
		\includegraphics[scale=0.22]{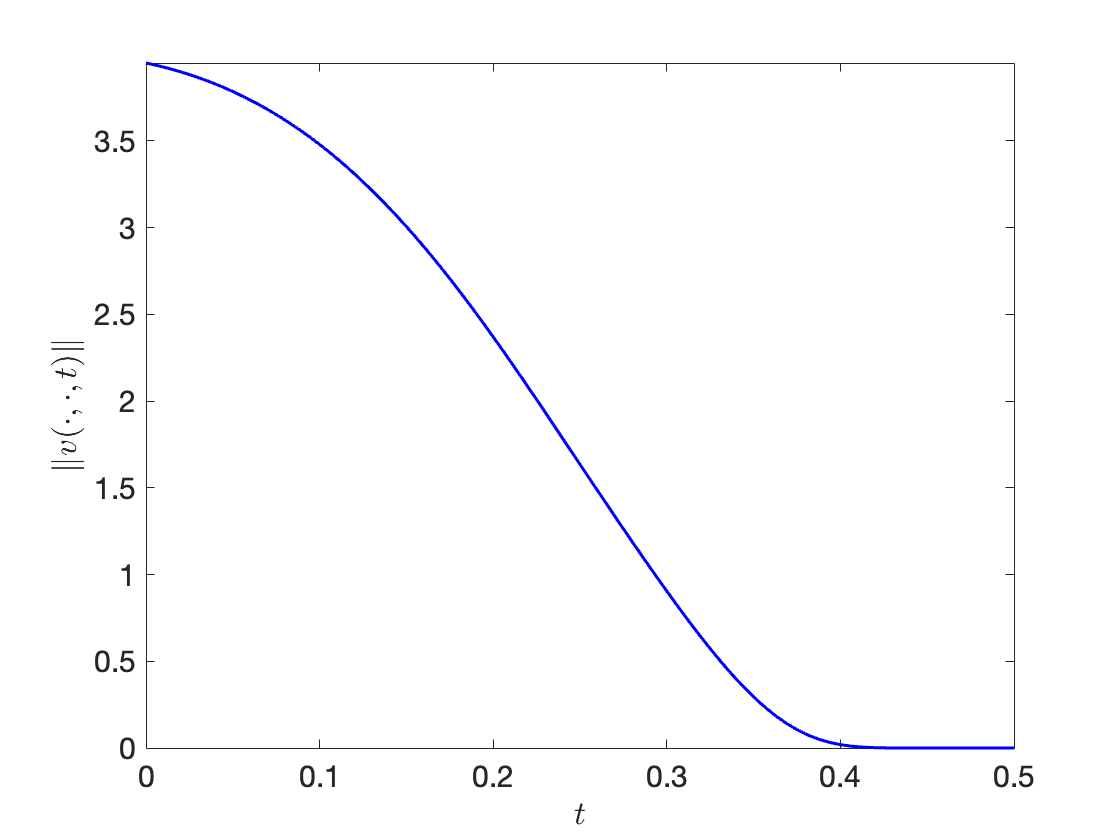}	
	\end{minipage}
	\vskip-.5cm
	\caption{The evolution of~$\|\bar{y}(\cdot\,,t)\|_{L^2}$ and~$\|y(\cdot\,,t)\|_{L^2}$ (left) and the evolution of~$\| v(\cdot\,,t) \|_{L^2(\om)}$ (right) as~$t \in [0,T]$.}
	\label{Figure:norm:ybar:y}
\end{figure}

\begin{figure}[!h]
	\centering 
	\includegraphics[scale=0.22]{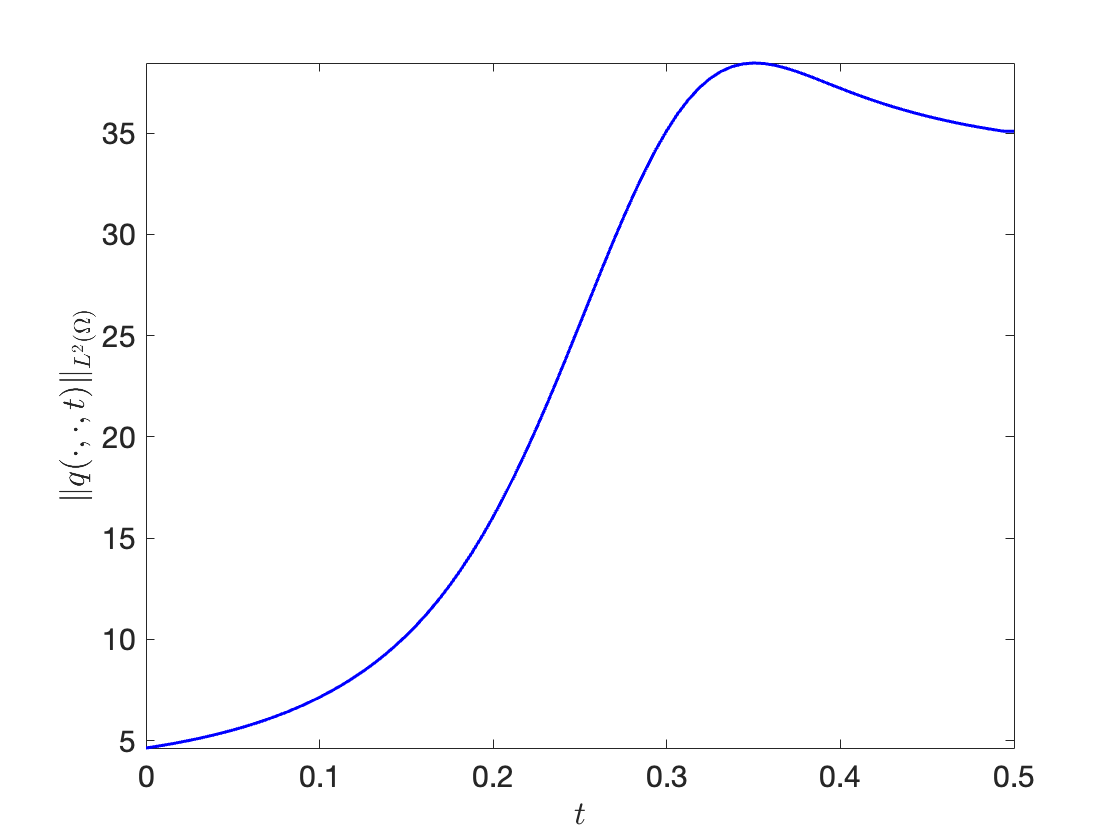}
	\vskip-.5cm
	\caption{The evolution of~$\|q(\cdot\,,t)\|_{L^2}$ with~$t\in [0,T]$.}
	\label{Figure:norm:computed_q}
\end{figure}

\begin{figure}[!h]
	\centering 
	\begin{minipage}[b]{0.30\textwidth}
            \includegraphics[scale=0.40]{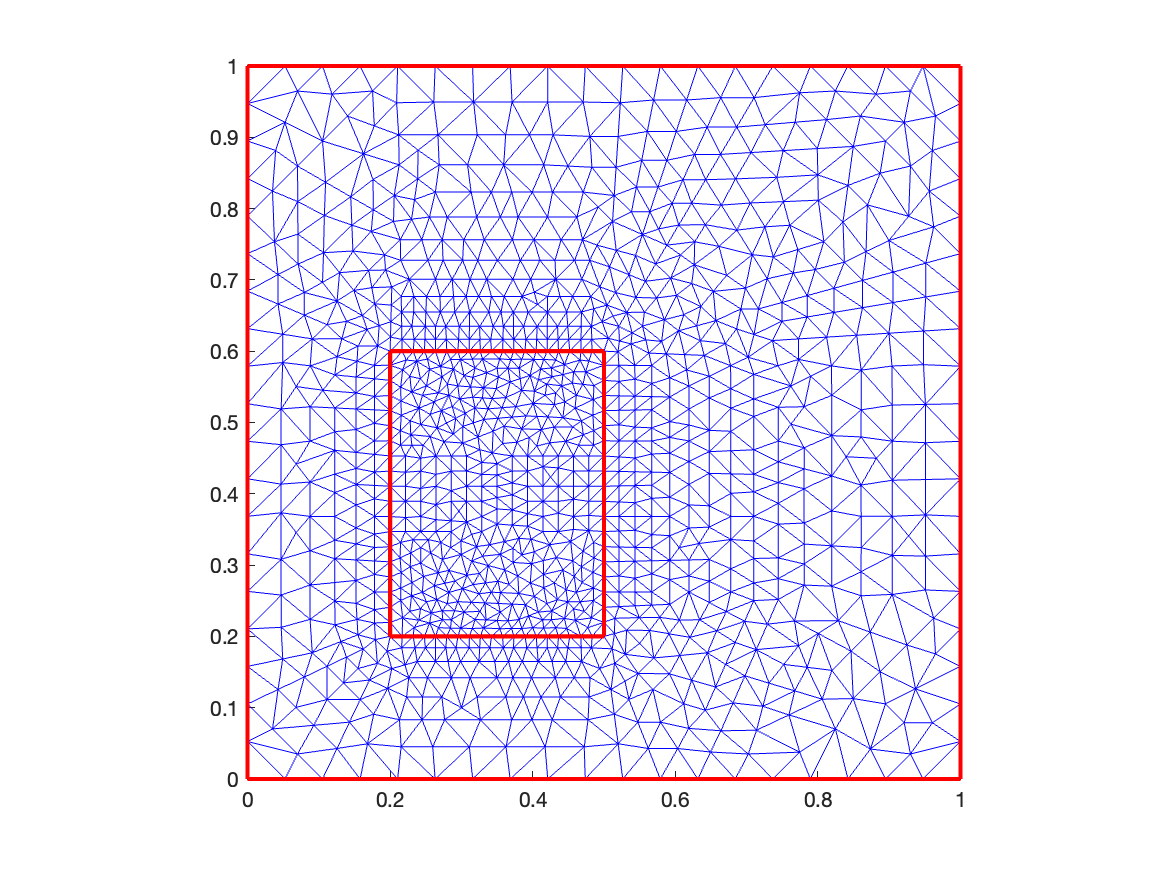}
    \end{minipage}
    \hfill
    \begin{minipage}[b]{0.5\textwidth}
            \includegraphics[scale=0.40]{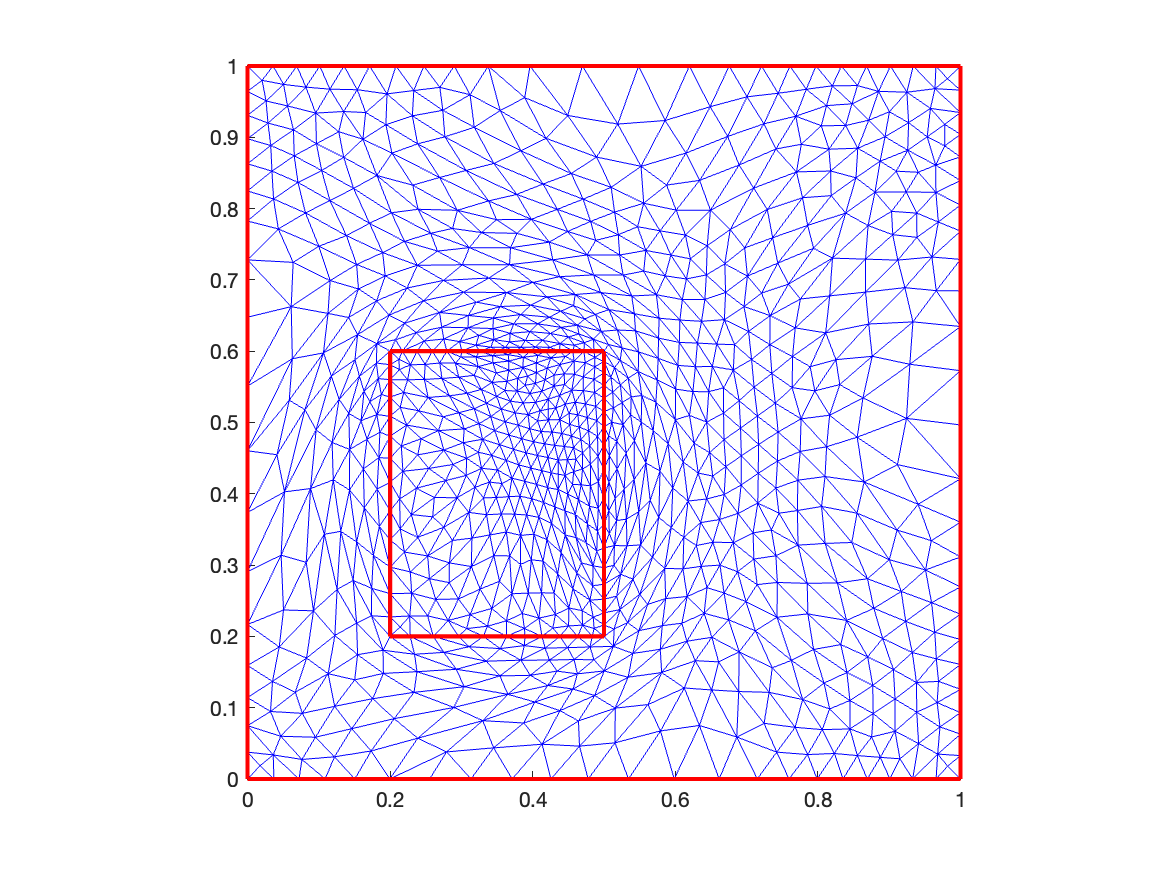}
    \end{minipage}
    \vskip-.5cm
    \caption{The initial(left) and final(right) meshes.}
    \label{Figure:Initial:final:mesh}
\end{figure}


\subsection{Test \#2: The behavior of the state and the control as~$R\to\infty$ for ALG~3}

   Let us fix~$\omega=(0.2,0.8) \times (0.2,0.8)$, $T=0.5$ and~$N_t=20$ and let us consider the corresponding system \eqref{numerical:experiences:eq:01}.
   We report the performance obtained by~ALG~3 with~$K=0.1$ by using different values of~$R$ and adapting the mesh each~$10$ iterates.

   In Table~\ref{Table:R:Euler:ALG3}, we give details on the behavior of the norms of~$y$, $v$, $q$, $y(\cdot\,,T)$ and~$q(\cdot\,,T)$ with respect to~$R$.
   We observe that~$R$ has a weak influence on the~$L^2$-norms of~$y$ and~$v$.
   This confirms that, as~$R$ increases, we get the uniform convergence of the state and the control.
   The same happens to the~$L^2$-norm of~$q$ and the norm in~$L^2(\Omega)$ of~$q(\cdot\,,T)$.

   On the other hand, the norm of~$\|y(\cdot\,,T)\|_{L^2}$ strongly depends on~$R$ and we see that it tends to 0 as~$R\to \infty$ in practice proportionally to $1/R$.    

\begin{table}[!h]
\centering 
\begin{tabular}{| c || c| c| c|c|c|} \hline 
$R$ &~$\|y\|_{L^2(Q)}$ &~$\|v\|_{L^2(Q_\omega)}$ &~$\|q\|_{L^2(Q)}$ &~$\|y(\cdot\,,T)\|_{L^2(\Omega)}$ &~$\|q(\cdot\,,T)\|_{L^2(\Omega)}$ \\ \hline \hline
$10^3$ &~$0.155245$ &~$1.47251$ &~$24.1339$ &~$1.33195\times 10^{-9}$ &~$8.55816$ \\
$10^4$ &~$0.155226$ &~$1.47274$ &~$24.1122$ &~$1.32407\times 10^{-11}$ &~$8.52156$ \\
$10^5$ &~$0.155262$ &~$1.47354$ &~$24.1811$ &~$1.319610^{-13}$ &~$8.62848$ \\
$10^6$ &~$0.155224$ &~$1.47355$ &~$24.1381$ &~$1.31352\times 10^{-15}$ &~$8.57142$ \\
$10^7$ &~$0.155211$ &~$1.47368$ &~$24.138$ &~$1.4735\times 10^{-17}$ &~$8.60645$ \\
$10^8$ &~$0.155186$ &~$1.47416$ &~$24.14$ &~$3.40873\times 10^{-19}$ &~$8.60533$ \\ \hline
\end{tabular}
\caption{The values of the norms of~$y$, $v$, $q$, $y(\cdot\,,T)$ and~$q(\cdot\,,T)$ as~$R$ increases.}
\label{Table:R:Euler:ALG3}
\end{table}

   The number of iterates needed to achieve
   \begin{align}
	\label{stopping:criteria}
	\dfrac{\|q^{k+1}-q^k\|_{L^2(Q)}}{\|q^k\|_{L^2(Q)}}\leq tol:=1\cdot 10^{-5}
   \end{align}
and the numbers of vertices and triangles in the final meshes are furnished in~Table~\ref{Table:R:iterates:mesh}.

\begin{table}[!h]
\centering 
\begin{tabular}{| c || c| c|c|} \hline
$R$ &~$\#$ iterates &~$\#$ vertices &~$\#$ triangles \\ \hline \hline 
$10^3$ &~$3035$ &~$1146$ &~$2154$\\ \hline 
$10^4$ &~$3037$ &~$1148$ &~$2159$\\ \hline 
$10^5$ &~$3028$ &~$1151$ &~$2166$\\ \hline 
$10^6$ &~$3030$ &~$1149$ &~$2160$\\ \hline 
$10^7$ &~$3026$ &~$1155$ &~$2174$\\ \hline 
$10^8$ &~$3025$ &~$1149$ &~$2161$\\ \hline 
\end{tabular}
\caption{The number of iterates, vertices, and needed triangles as $R$ increases.}
\label{Table:R:iterates:mesh}
\end{table}

   We observe that these values remain stable as~$R$ increases.
   This indicates that ALG~3 behaves robustly ``with respect to truncation''.
   Moreover, we observe that the huge number of iterates to achieve~\eqref{stopping:criteria} is similar to those obtained by the conjugate gradient method applied in~\cite{EFC2014Numerical} due to the lack of uniform coercivity of the dual problem.

   Tables~\ref{Table:R:Euler:ALG3} and~\ref{Table:R:iterates:mesh} suggest that it is not necessary to take~$R \to +\infty$ to achieve a good approximation of the controls.
   In fact, thanks to the use of the weight functions~$\rho_R$ and~$\rho_0$, the norms of the computed controls and controlled states change only slightly concerning these parameters.
   Moreover, in contrast to the case of conjugate gradient algorithm, we do not have a significative increase of iterates when~$R$ increases.


\subsection{Test \#3: Influence of the weights}



   Let us take~$\omega=(0.2,0.8)\times (0.2,0.8)$ and~$T=0.5$ and consider again the controlled problem~\eqref{numerical:experiences:eq:01}.
   We apply again ALG~3 with $K=0.1$, $R=10^5$, $q^0\equiv 0$, adapting the spatial mesh every 10 iterates. In this case, the stopping criteria is as \eqref{stopping:criteria} with $tol=10^{-4}$. 

   In~Tables~\ref{Table:Influence:Nt:01} and~\ref{Table:Influence:Nt:02}, the behavior of the algorithm is shown for various values of~$N_t$.

\begin{table}[!h]
	\centering
	\begin{tabular}{|c|| c| c| c|c|c|} \hline 
		$N_{t}$ &~$\|y\|_{L^2(Q)}$ &~$\|v\|_{L^2(Q_\omega)}$ &~$\|q\|_{L^2(Q)}$ &~$\|y(\cdot\,,T)\|_{L^2}$ &~$\|q(\cdot\,,T)\|_{L^2}$\\ \hline \hline 
		20 & 0.156718 & 1.17555 & 5.3525 &~$2.63373\cdot 10^{-13}$ & 6.18708\\ \hline  
		40 & 0.150623 & 1.1082 & 4.29809 &~$2.78415\cdot 10^{-13}$ & 5.84151\\ \hline 
		60 & 0.147944 & 1.08737 & 3.97457 &~$2.96496\cdot 10^{-13}$ & 5.7084\\ \hline
		80 &   0.146301 & 1.07509 & 3.81718 &~$3.14825\cdot 10^{-13}$ & 5.6242\\ \hline 
		100 & 0.145286 & 1.06806 & 3.73284 &~$3.28547\cdot 10^{-13}$ & 5.5786\\ \hline 
		120 & 0.144542 & 1.06397 & 3.68143 &~$3.36813\cdot 10^{-13}$ & 5.55647\\ \hline 
		160 & 0.143917 & 1.06091 & 3.63826 &~$3.38874\cdot 10^{-13}$ & 5.52634\\ \hline 
		200 & 0.142921 & 1.05326 & 3.57356 &~$3.61723\cdot 10^{-13}$ & 5.48535\\ \hline 
	\end{tabular}
	\caption{The norms of~$y$, $v$, $q$, $y(\cdot\,,T)$ and~$q(\cdot\,,T)$ as~$N_t$ increases.}
	\label{Table:Influence:Nt:01}
\end{table}

\begin{table}[!h]
	\centering 
	\begin{tabular}{|c|| c|c|c|} \hline 
		$N_{t}$ & \#iterates & \#vertices & \#triangles\\ \hline 
		20 & 462 & 887 & 1657 \\ \hline 
		40 & 379 & 861 & 1608 \\ \hline
		60 & 342 & 840 & 1569 \\ \hline 
		80 & 321 & 827 & 1544 \\ \hline 
		100 & 309 & 813 & 1518 \\ \hline 
		120 & 301 & 803 & 1497 \\ \hline 
		160 & 294 & 801 & 1494 \\ \hline 
		200 & 284 & 792 & 1478 \\ \hline  
	\end{tabular}
	\caption{The number of iterates, vertices, and triangles of the final mesh when~$N_t$ increases.}
	\label{Table:Influence:Nt:02}
\end{table}

   In particular, the uniform convergence of the control and the controlled state is clear.
   Moreover, the control functions approximate satisfactorily the null controllability requirement, since in each case~$\|y(\cdot\,,T)\|_{L^2}\sim 10^{-13}$.
   The number of iterates is slightly reduced when~$N_t$ increases and the final mesh obtained in each case remains without significative changes.

   To end the experiments for the 2D heat equation using ALG3, we study the evolution of the $\log_{10}$ scale of the relative error to achieve the condition
\begin{align*}
	\dfrac{\|q^{k+1}-q^k\|_{L^2(Q)}}{\|q^k\|_{L^2(Q)}}\leq 10^{-6}. 
\end{align*}

This is displayed in the Figure \ref{graph:log:error}.

\begin{figure}[!h]
	\centering 
	\includegraphics[scale=0.25]{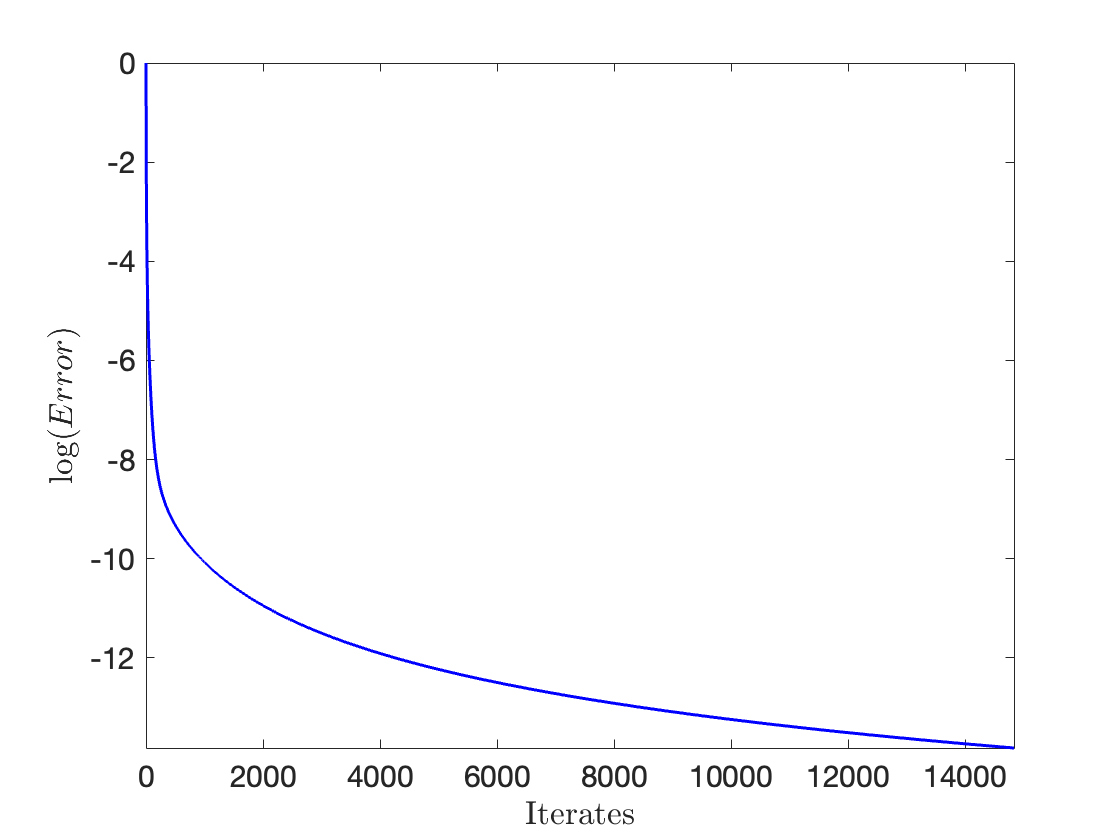}
	\vskip-.5cm
	\caption{Iterates vs Logarithm of the relative error.}
	\label{graph:log:error}
\end{figure}

   We thus see that the evolution of the relative error is nonlinear with respect to the number of iterates, a usual phenomenon in ill-posed parabolic problems. To be more precise, the slope of the curve reduces significantly after the first 200 iterates.
   

\subsection{Test \#4: An experiment for a $3$D heat equation}

   Now, we consider the~$3$D domain~$\Omega=[0,1]^3$ and the control region~$\omega=[0.2,0.4]^3$, we fix~$T=0.5$ and we deal with the (controlled) heat equation~\eqref{eq:main:01} with~$a=0.01$, $b=-1.0$ and the initial condition
\begin{align*}
	y^0(x_1,x_2,x_3)=100\cdot\sin(\pi x_1) \, \sin(\pi x_2) \, \sin(\pi x_3) \quad \forall (x_1,x_2,x_3)\in \Omega.
\end{align*}
   We also consider a 3D version of the function~$\eta^0$ given by~\eqref{numerical:eta:0} and set the functions~$\rho$ and~$\rho_0$ accordingly.
   The goal is to solve the extremal problem~\eqref{Extremal:problem:01}.

   We consider a 3D mesh whose numbers of tetrahedrons and vertices are, respectively, $11023$ and~$2042$; see Figure~\ref{Fig:3Dmesh}.

\begin{figure}[!h]
	\centering 
	\includegraphics[scale=0.5]{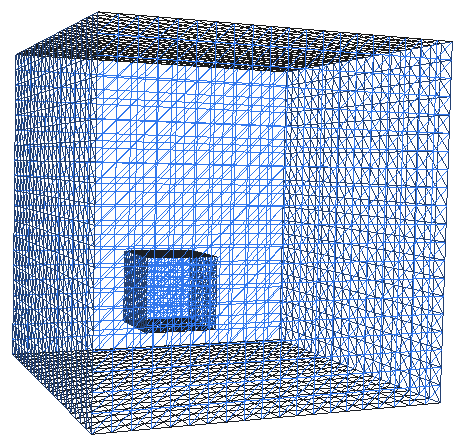}
	\vskip-.5cm
	\caption{3D mesh with~$\#t=11023$ and~$\#v=2042$.}
	\label{Fig:3Dmesh}
\end{figure}  
%

   We use ALG~3 to solve the null controllability problem.
   The stopping criteria is 
\begin{align*}
	\dfrac{\|q^{k+1}-q^{k}\|_{L^2(Q)}}{\|q^k\|_{L^2(Q)}}\leq 10^{-5}. 
\end{align*}


   The computed controlled solution can be found in~Figure~\ref{3D:evolution:computed:state} at several times.
   On the other hand, the projected controlled state and the projected associated control at~$x_1=x_2=0.3$ are given in~Figures~\ref{Fig:3D-proj-state-y} and~\ref{Fig:3D-proj-control-v}.

\begin{figure}[!h]
	\begin{minipage}[b]{0.30\textwidth}
		\centering 
    	\includegraphics[scale=0.15]{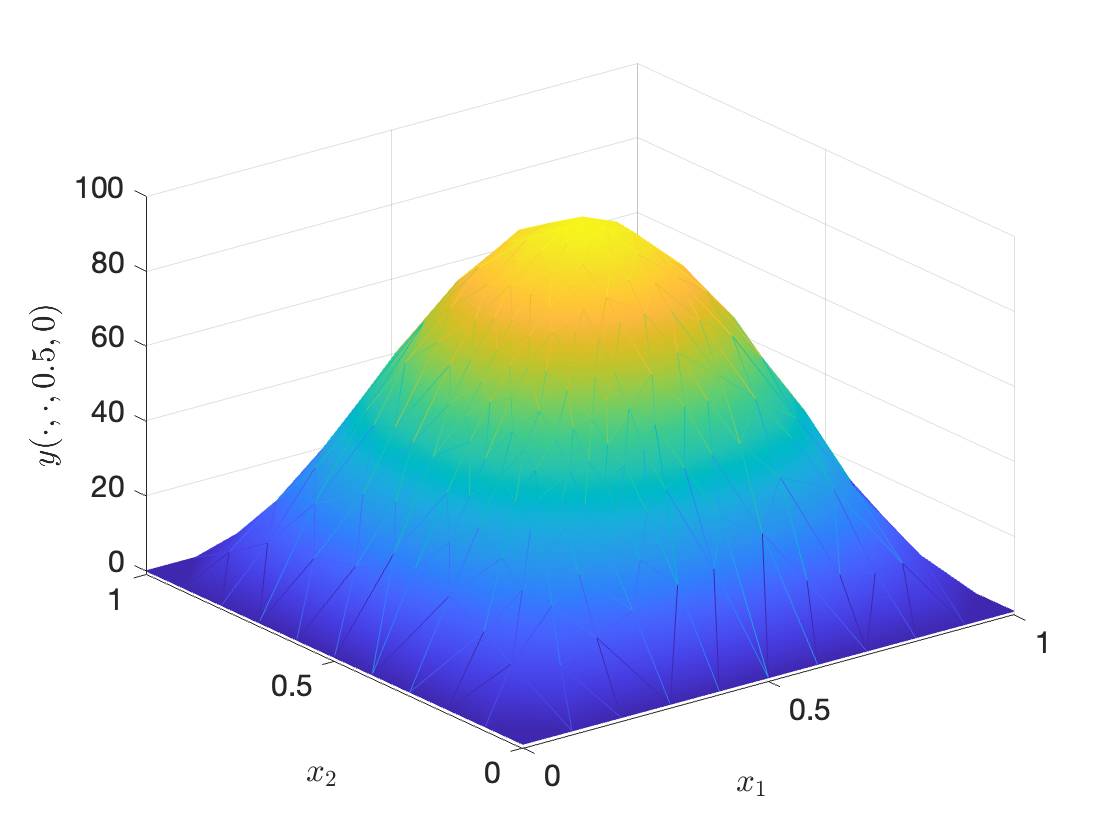}
    \end{minipage}
    \hfill
    \begin{minipage}[b]{0.30\textwidth}
    	\centering 
    	\includegraphics[scale=0.15]{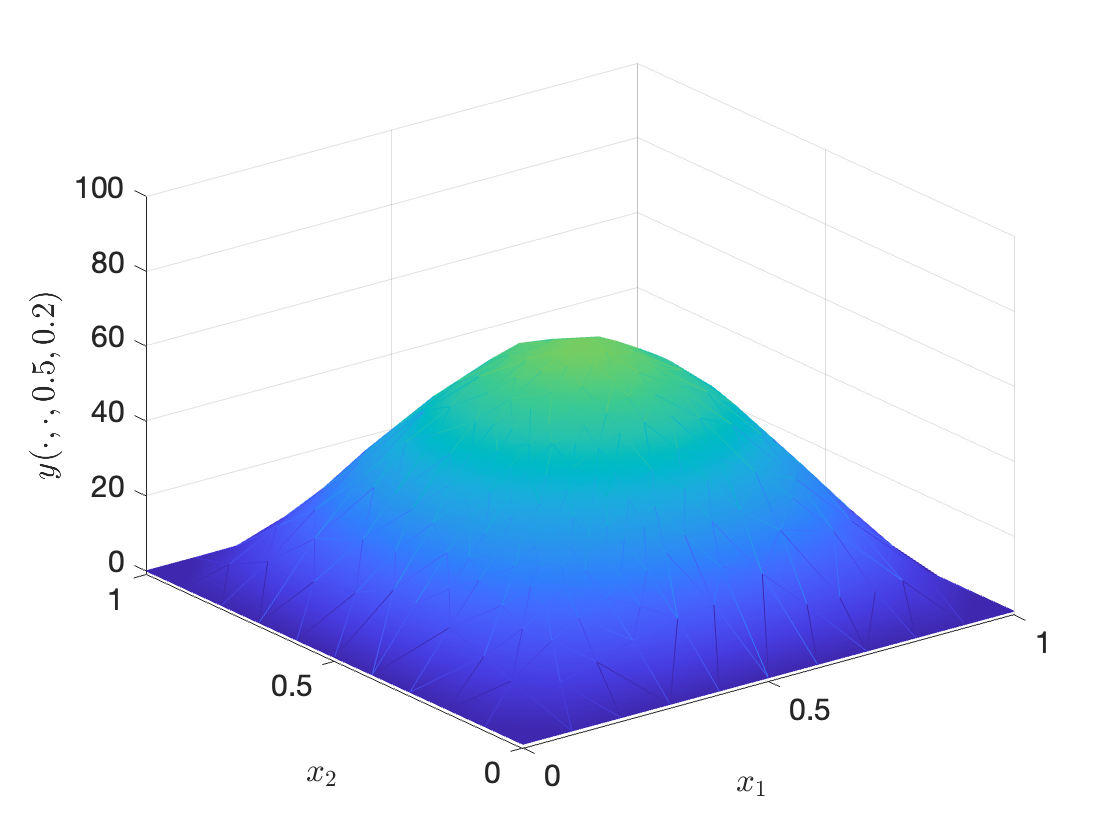}
    \end{minipage}
    \hfill
    \vspace{1em}
    \begin{minipage}[b]{0.30\textwidth}
    	\centering 
    	\includegraphics[scale=0.15]{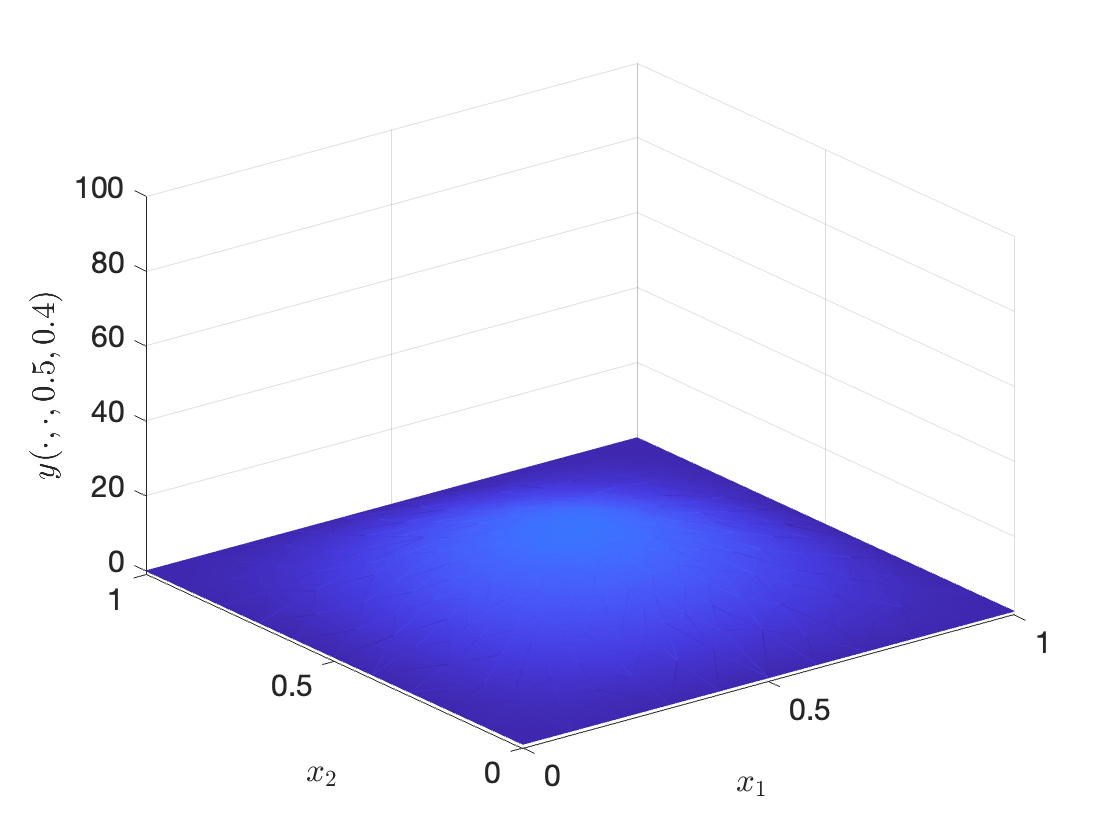}
    \end{minipage}
 	\vskip-.5cm
	\caption{Evolution of the computed state at~$t=0$ (left), $0.2$ (center) and~$0.4$ (right).}
	\label{3D:evolution:computed:state}    
\end{figure}

\begin{figure}[!h]
	\begin{minipage}[b]{0.50\textwidth}
		\centering 
    	\includegraphics[scale=0.40]{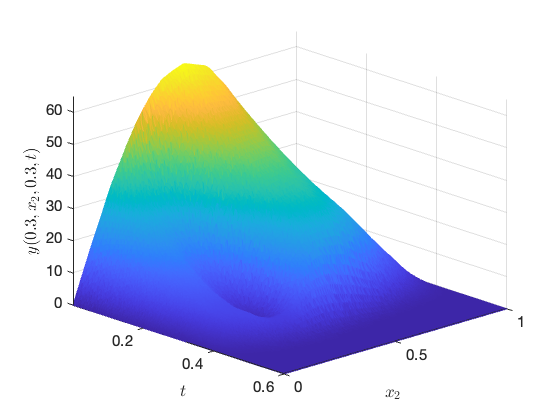}
	\vskip-.5cm
	\caption{The projected state ($x_1\!=\!x_3\!=\!0.3$).}
	\label{Fig:3D-proj-state-y}
    \end{minipage}
    \hfill
    \begin{minipage}[b]{0.5\textwidth}
    	\centering 
    	\includegraphics[scale=0.40]{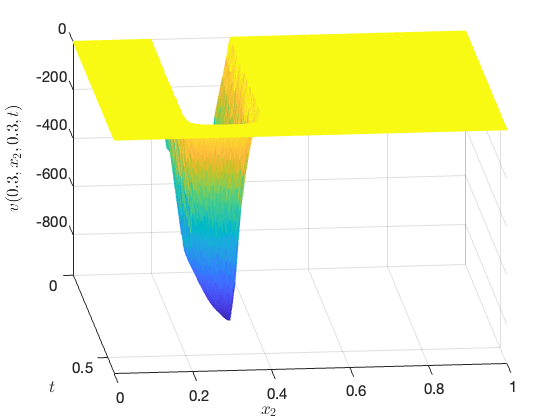}
	\vskip-.5cm
	\caption{The projected control ($x_1\!=\!x_3\!=\!0.3$).}
	\label{Fig:3D-proj-control-v}
    \end{minipage}
\end{figure}

%

   In addition, we have 
\begin{align*}
	\|y(\cdot\,,T)\|_{L^2}=2.09824\cdot 10^{-7}\text{ and }\|q(\cdot\,,T)\|_{L^2}=45.8320.
\end{align*}
   The evolution in time of~$\|\bar{y}(\cdot\,,t)\|_{L^2}$, $\|y(\cdot\,,t)\|_{L^2}$ and~$\|v(\cdot\,,t)\|_{L^2(\omega)}$ are depicted in Figures~\ref{Fig:3D:norm:y:} and~\ref{Fig:3D:norm:v}.
   In addition, the evolution in time of~$\|q(\cdot\,,t)\|_{L^2}$ is presented in~Figure~\ref{Fig:3D:norm:q}.

\begin{figure}[!h]
	\begin{minipage}[b]{0.50\textwidth}
		\centering 
    	\includegraphics[scale=0.40]{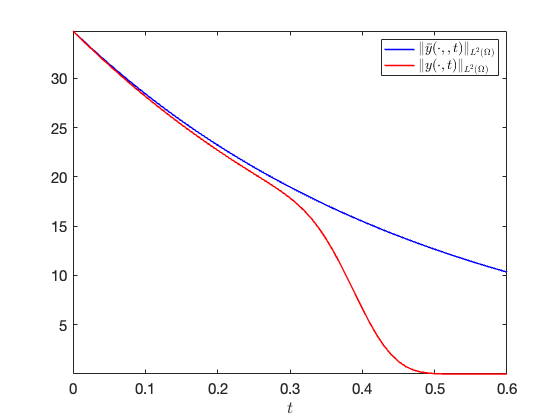}
	    \vskip-.5cm
    	\caption{$\|\bar{y}(\cdot\,,t)\|_{L^2}$ and~$\|y(\cdot\,,t)\|_{L^2}$.}
    	\label{Fig:3D:norm:y:}
    \end{minipage}
    \hfill
    \begin{minipage}[b]{0.5\textwidth}
    	\centering 
    	\includegraphics[scale=0.40]{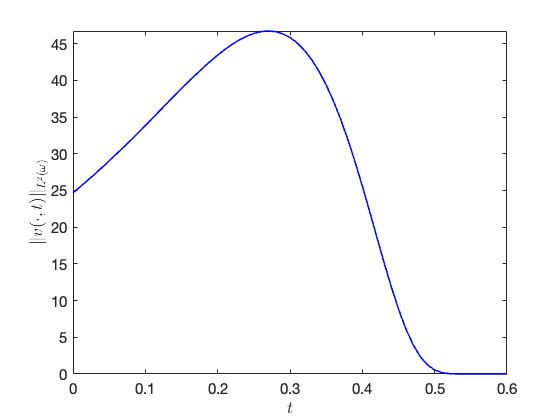}
	    \vskip-.5cm
    	\caption{Evolution of~$\|v\|_{L^2(Q_\omega)}$.}
    	\label{Fig:3D:norm:v}
    \end{minipage}
\end{figure}

\begin{figure}[!h]
	\centering 
	\includegraphics[scale=0.4]{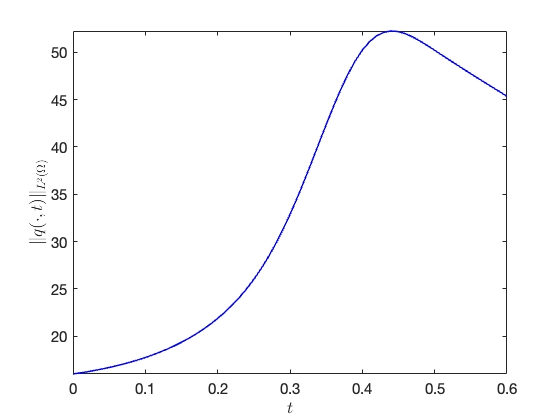}
	\vskip-.5cm
	\caption{Evolution of~$\|q\|_{L^2(Q)}$.}
	\label{Fig:3D:norm:q}
\end{figure}

\subsection{Test \#5: An experiment for the $2$D Stokes system}

   In this experiment, we take~$\Omega=(0,1)^2\subset \mathbb{R}^2$, $\omega=(0.2,0.4)^2$ and~$T=0.6$.
   We consider the problem~\eqref{Stokes:system:01} with~$a=0.05$,
\begin{align*}
	\bm{x}=\left(
	\begin{array}{c}
		x_1\\x_2
	\end{array}
	\right),\quad 
	\bm{y}=\left(
	\begin{array}{c}
		y_1\\y_2 
	\end{array}
	\right),\quad \bm{v}=\left(
	\begin{array}{c}
		v_1\\v_2
	\end{array}
	\right),
\end{align*}
with the initial condition
\begin{align*}
	\bm{y}^0(\bm{x}) = 10^3\left( 
	\begin{array}{c}
		x_1^2(1-x_1)^2x_2(1-x_2)^2\\
		-x_1(1-x_1)(0.5-2x_1)x_2^2(1-x_2)^2.
	\end{array}
	\right).
\end{align*}

   We consider a mesh with~$1698$ triangles and~$878$ vertices (see~Figure~\ref{Fig-Stokes-mesh}).

\begin{figure}[!h]
	\centering 
	\includegraphics[scale=0.5]{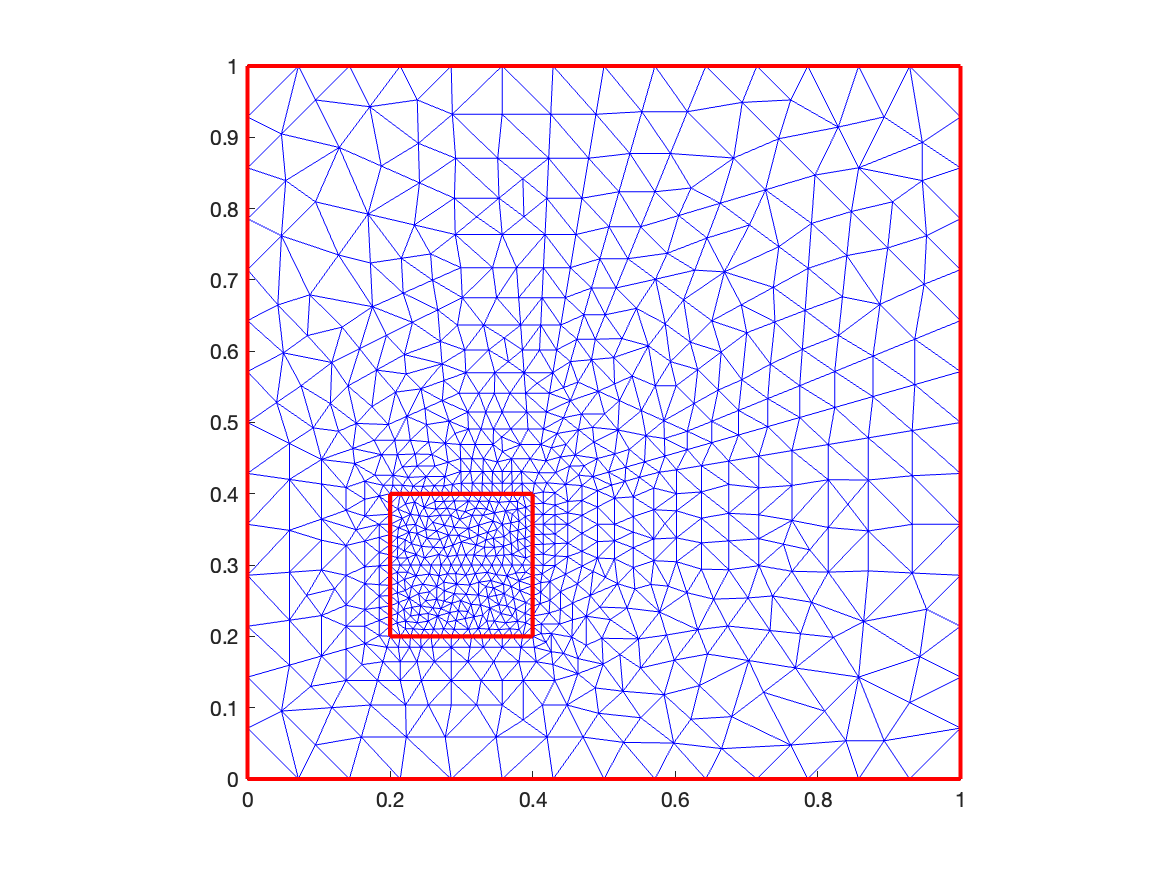}
	\vskip-.5cm
	\caption{Mesh with 1698 triangles and 878 vertices.}
	\label{Fig-Stokes-mesh}
\end{figure} 

   The uncontrolled solution of the Stokes system is depicted in~Figure~\ref{Stokes-uncontrolled-solution}.
   We note that the solution does not vanish at~$t=T$.
   We apply ALG~3 with~$K=0.1$, $R=10^{5}$ and stopping criteria
\begin{align*}
	\frac{\|\bm{q}^{k+1}-\bm{q}^k\|_{\bm{L}^2(Q)}}{\|\bm{q}^k\|_{\bm{L}^2(Q)}}\leq 1\cdot 10^{-5}.
\end{align*}

\begin{figure}[!h]
	\begin{minipage}[b]{0.30\textwidth}
	\centering 
    	\includegraphics[scale=0.22]{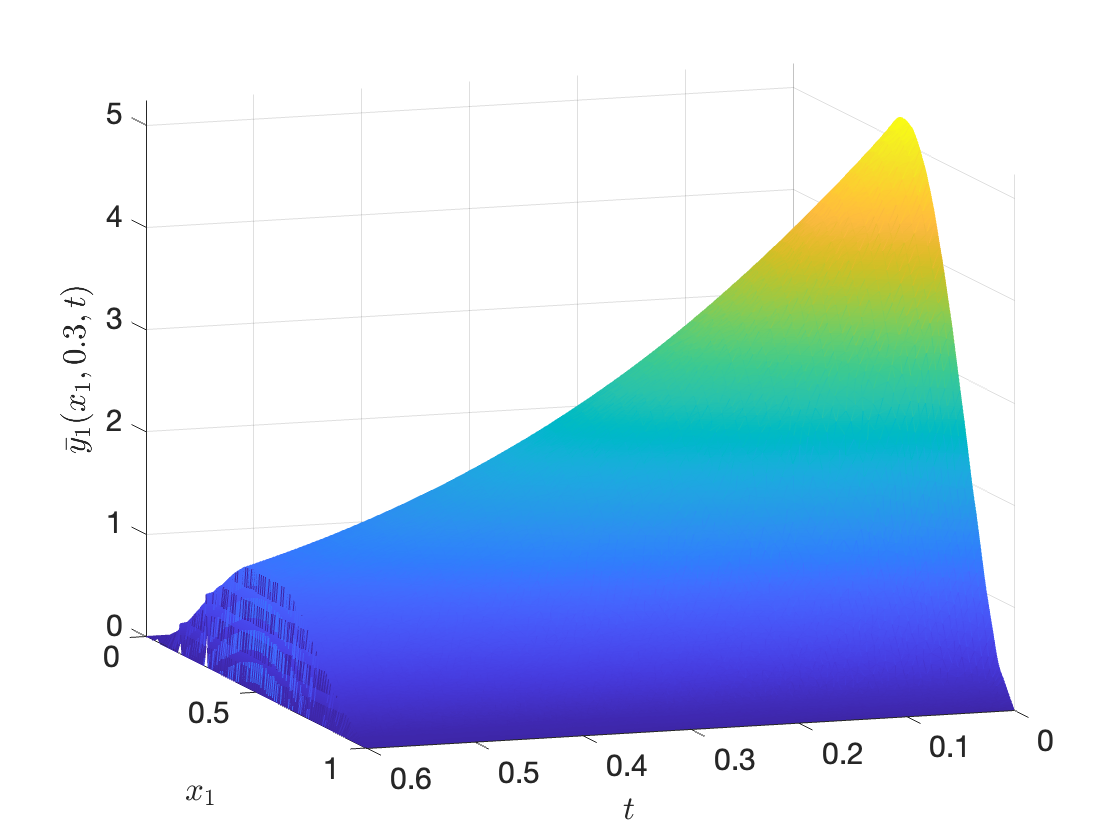}
    \end{minipage}
    \hfill
    \begin{minipage}[b]{0.5\textwidth}
    	\includegraphics[scale=0.22]{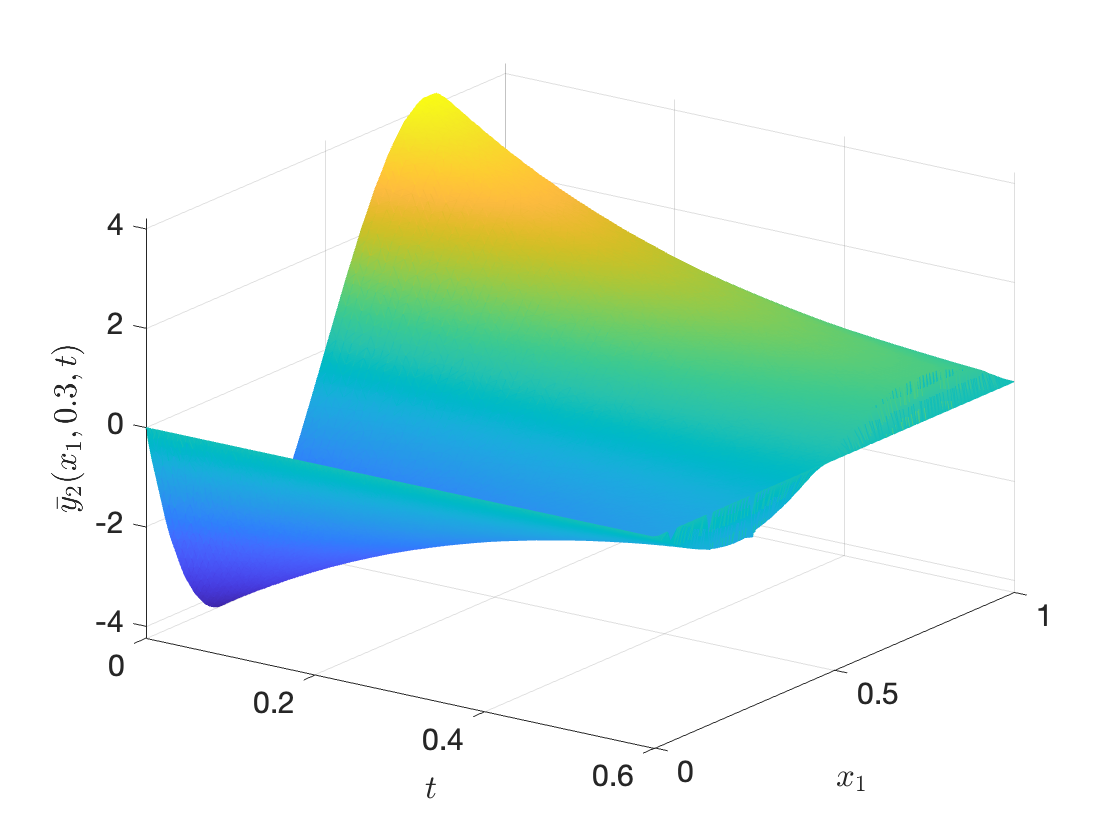}
    \end{minipage}
    \vskip-.5cm
    \caption{The~$x_1$-projected first(left) and second(right) component of the uncontrolled solution at~$x_2=0.3$.}
    \label{Stokes-uncontrolled-solution}
\end{figure}

\begin{figure}[!h]
	\centering 
	\begin{minipage}[b]{0.30\textwidth}
    	\includegraphics[scale=0.15]{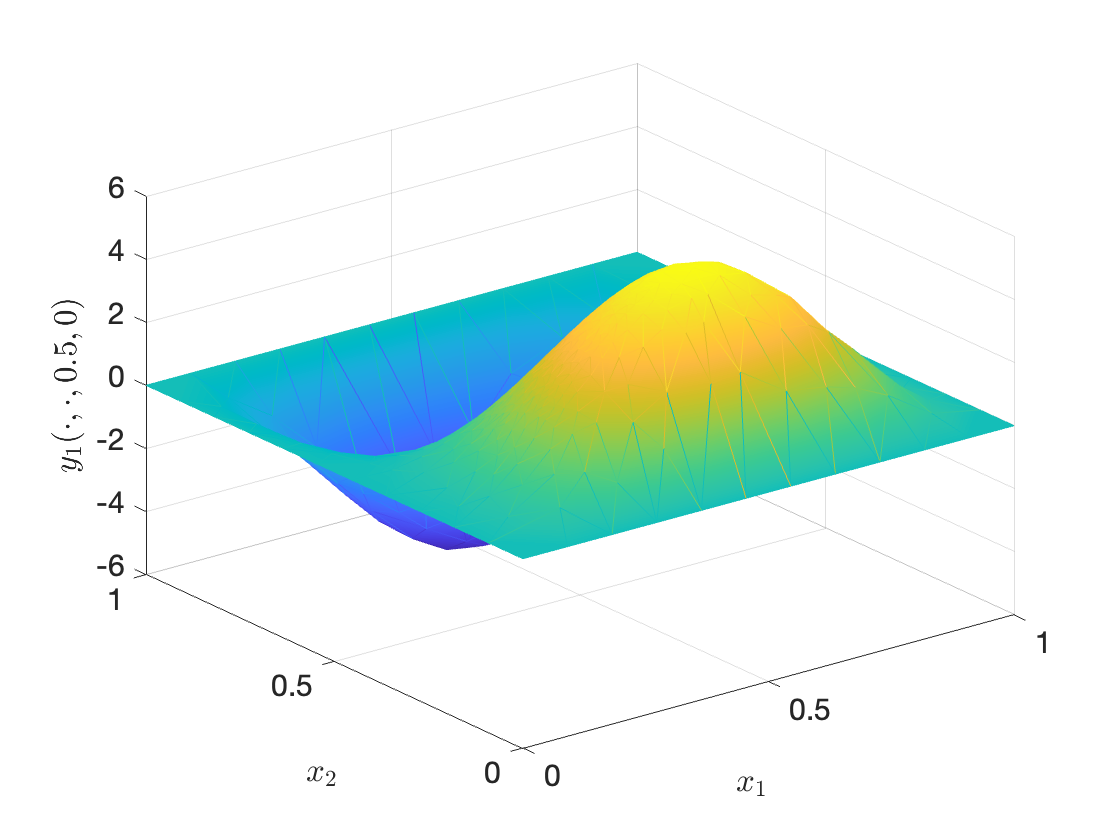}
    \end{minipage}
    \hfill
    \begin{minipage}[b]{0.30\textwidth}
    	\includegraphics[scale=0.15]{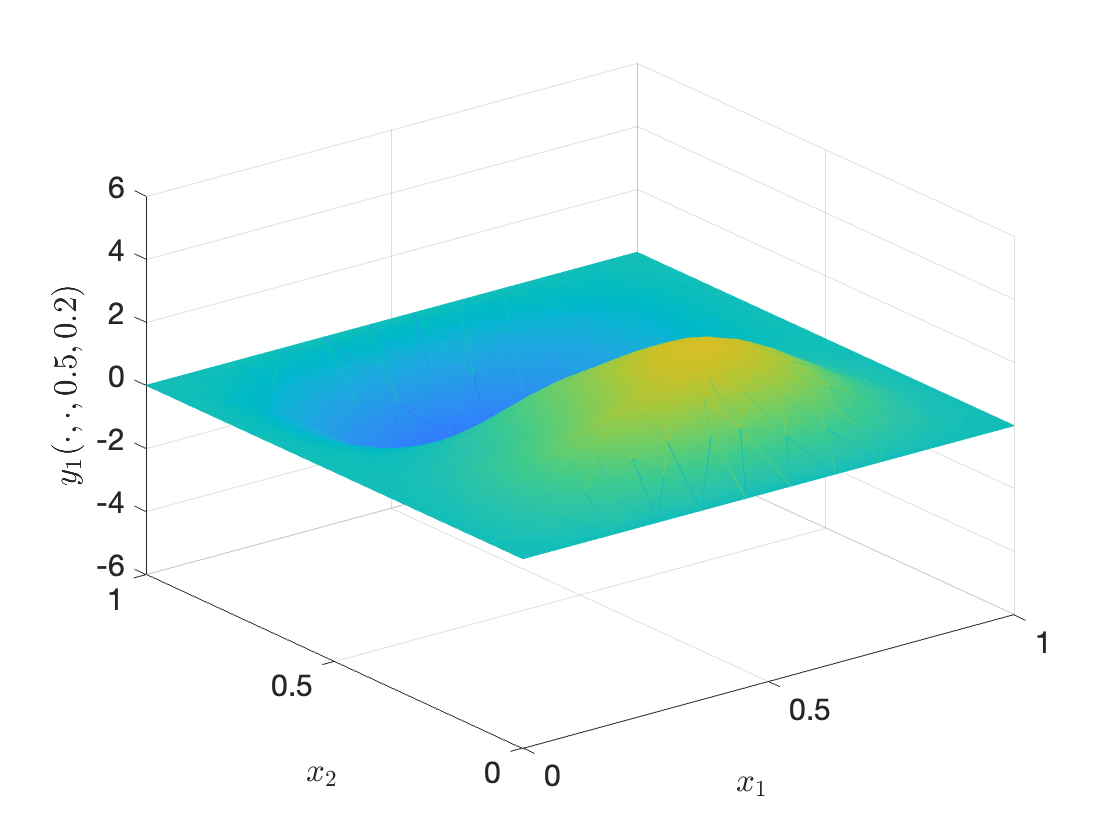}
    \end{minipage}
    \hfill 
    \begin{minipage}[b]{0.30\textwidth}
    	\includegraphics[scale=0.15]{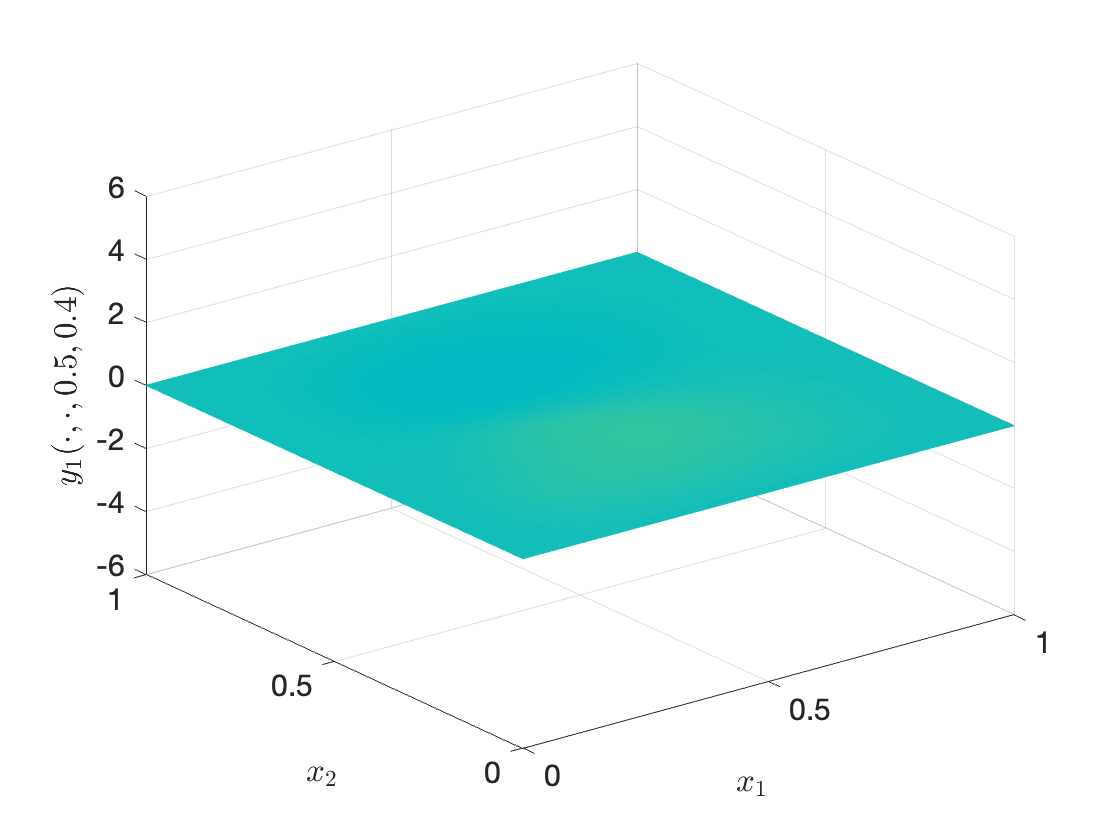}
    \end{minipage}    
    \caption{The projected controlled solution component $y_1$ in at~$t=0$(left), $t=0.2$(center) and $t=0.4$(right).}
    \label{Fig-Stokes-controlled-solution-y1}
\end{figure}

\begin{figure}[!h]
	\begin{minipage}[b]{0.30\textwidth}
    	\includegraphics[scale=0.15]{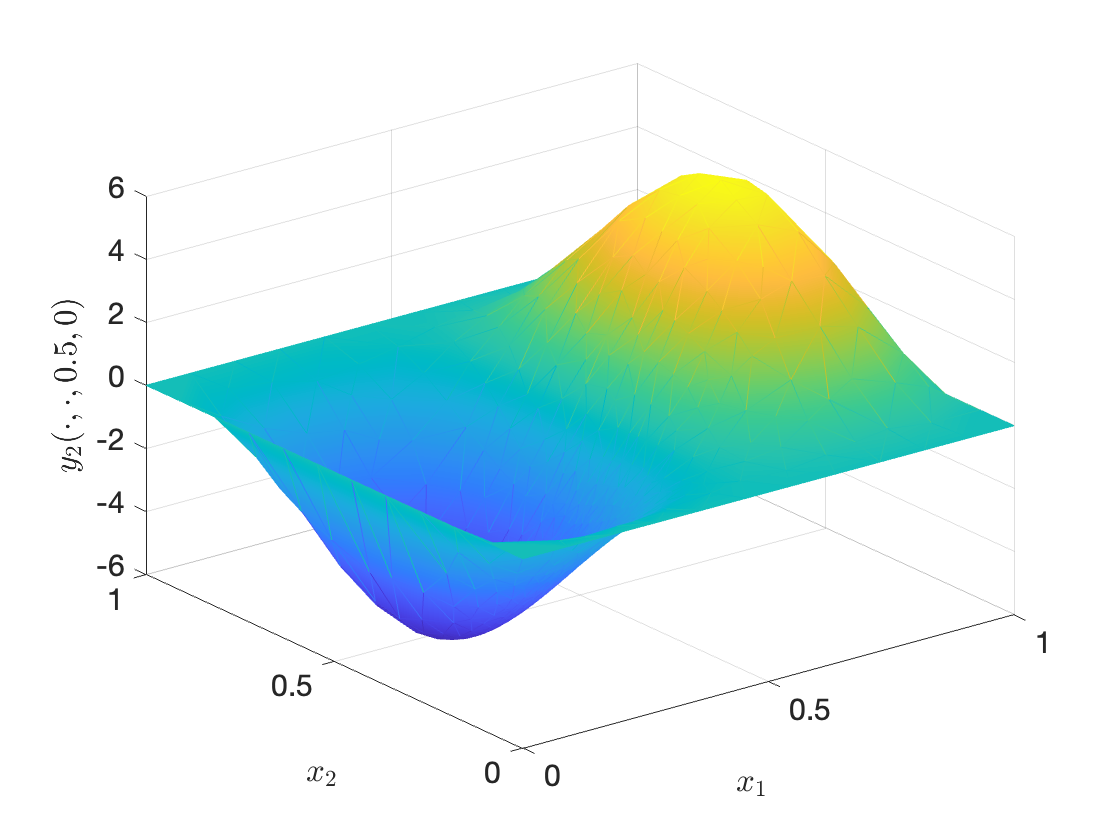}
    \end{minipage}
    \hfill
    \begin{minipage}[b]{0.30\textwidth}
    	\includegraphics[scale=0.15]{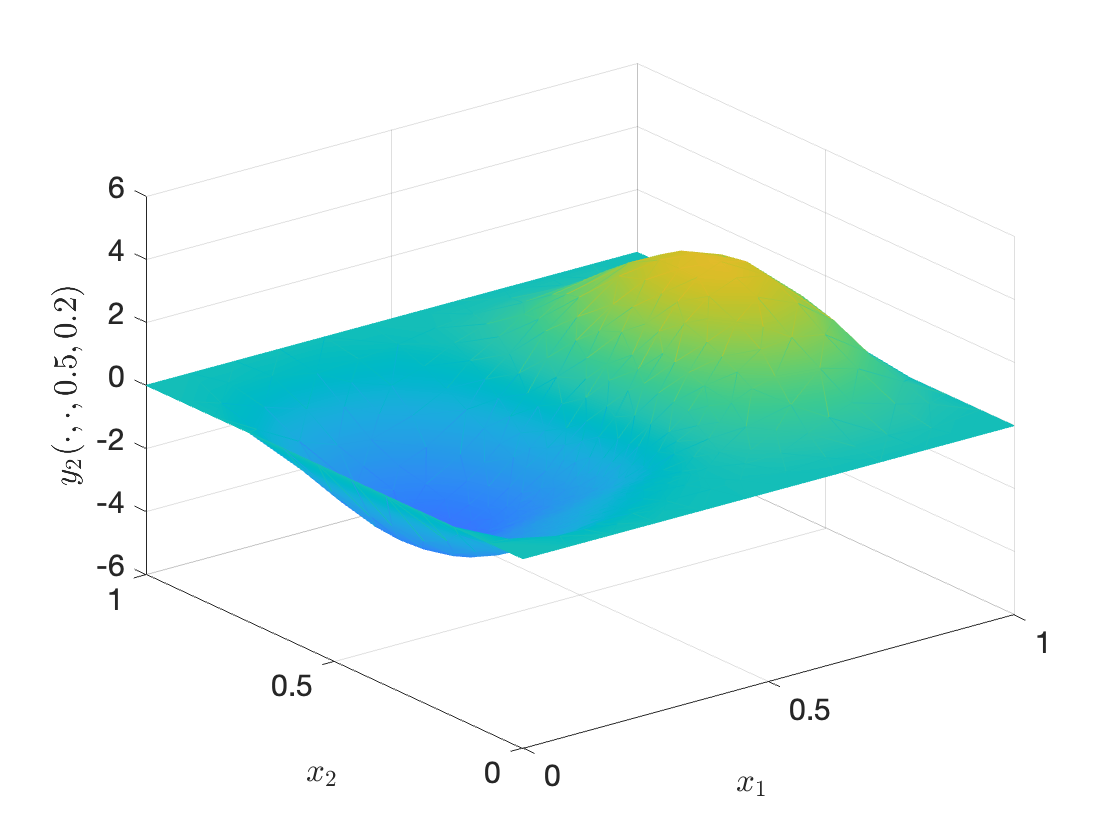}
    \end{minipage}
    \hfill 
    \begin{minipage}[b]{0.30\textwidth}
    	\includegraphics[scale=0.15]{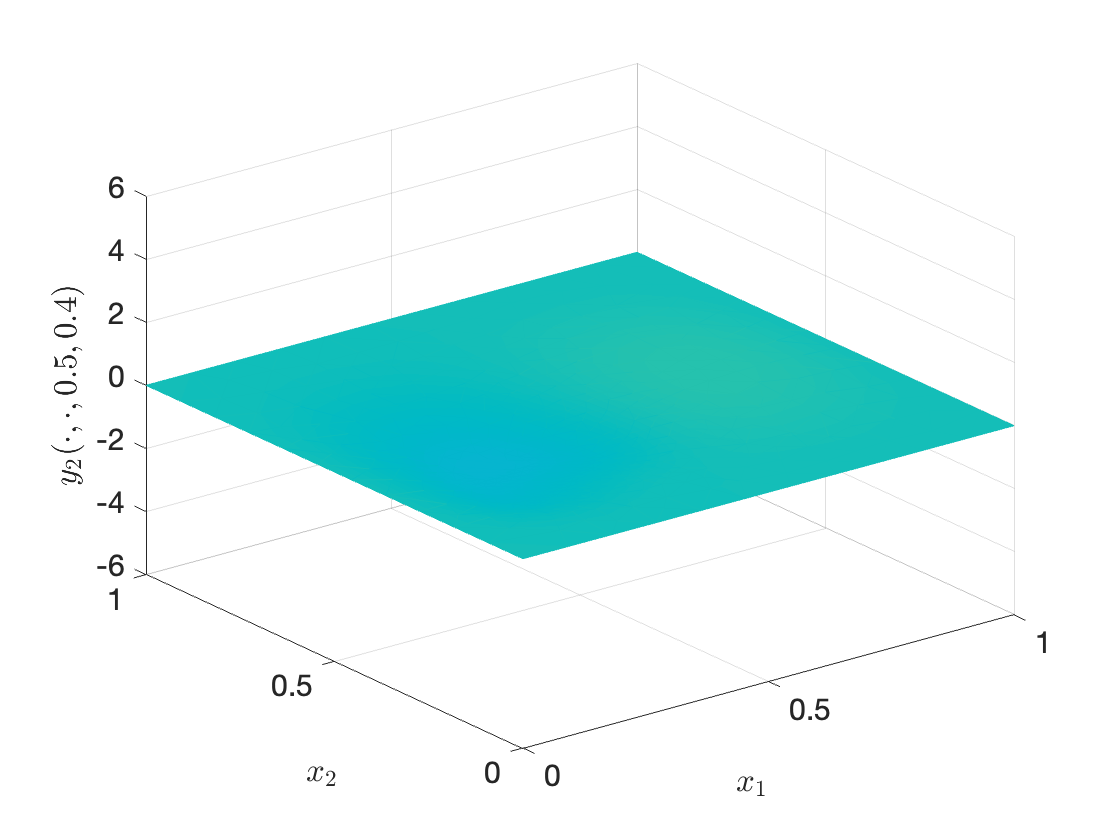}
    \end{minipage}    
    \caption{The projected controlled solution component $y_2$ in at~$t=0$(left), $t=0.2$(center) and $t=0.4$(right).}
    \label{Fig-Stokes-controlled-solution-y2}
\end{figure}
%
%
%


   The components of the projected state are given in~Figures~\ref{Fig-Stokes-controlled-solution-y1} and~\ref{Fig-Stokes-controlled-solution-y2}.
   We observe there that the solution vanishes at~$t=T$.
   In fact, the norms of the controlled state and~$\bm{q}$ at~$t=T$ are given by
  \begin{align*}
	\|\bm{y}(\cdot\,,T)\|_{L^2}=5.50185\times 10^{-9}
	\text{ and }\|\bm{q}(\cdot\,,T)\|_{L^2}=55.0852.
  \end{align*}

   The components of the projected control are depicted in~Figures~\ref{Fig:Stokes:evol:ybarvsy} and~\ref{Fig:Stokes:evol:control}.
   
    The evolution in time of the norms of~$\bar{\bm{y}}$, $\bm{y}$ and~$\bm{v}$ is given in~Figures~\ref{Fig:Stokes-ybary} and~\ref{Fig:Stokes-vnorm}.
    In addition, the evolution of~$\|\bm{q}(\cdot\,,t)\|_{L^2}$ is shown in~Figure~\ref{Fig:Stokes:qnorm}.

   Once more, we observe in~Figures~\ref{Fig:Stokes-ybary} and~\ref{Fig:Stokes-vnorm} that the norms of the state and the control go to zero as~$t\to T^-$.

\begin{figure}[!h]
	\begin{minipage}[b]{0.50\textwidth}
		\centering 
    	\includegraphics[scale=0.40]{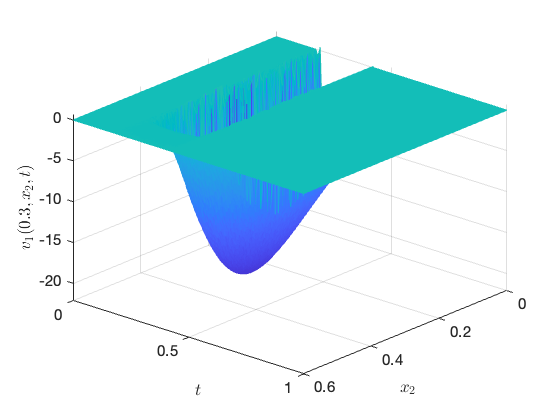}
	    \vskip-.5cm
    	\caption{The projected~$v_1$.}
    	\label{Fig:Stokes:evol:ybarvsy}
    \end{minipage}
    \hfill
    \begin{minipage}[b]{0.5\textwidth}
    	\centering 
    	\includegraphics[scale=0.40]{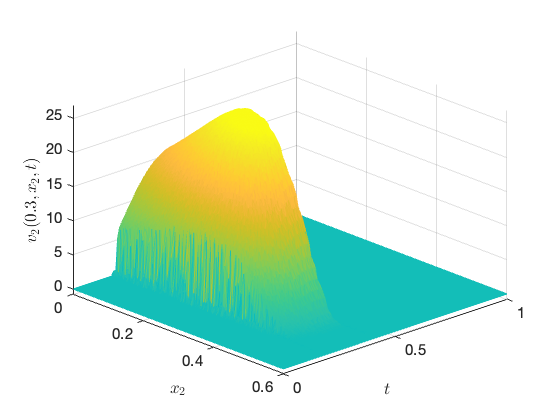}
	    \vskip-.5cm
    	\caption{The projected~$v_2$.}
    	\label{Fig:Stokes:evol:control}
    \end{minipage}    
\end{figure}

\begin{figure}[!h]
	\begin{minipage}[b]{0.50\textwidth}
		\centering 
    	\includegraphics[scale=0.44]{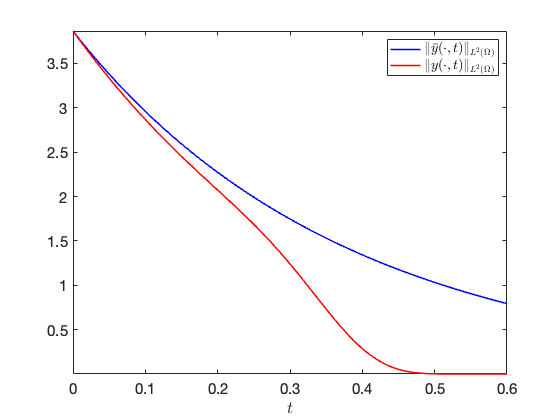}
	    \vskip-.5cm
	    \caption{The $L^2$-norms of $\bm{y}(\cdot,t)$ and $\bm{\overline{y}}(\cdot,t)$.}
    	\label{Fig:Stokes-ybary}
    \end{minipage}
    \hfill
    \begin{minipage}[b]{0.5\textwidth}
    	\centering 
    	\includegraphics[scale=0.22]{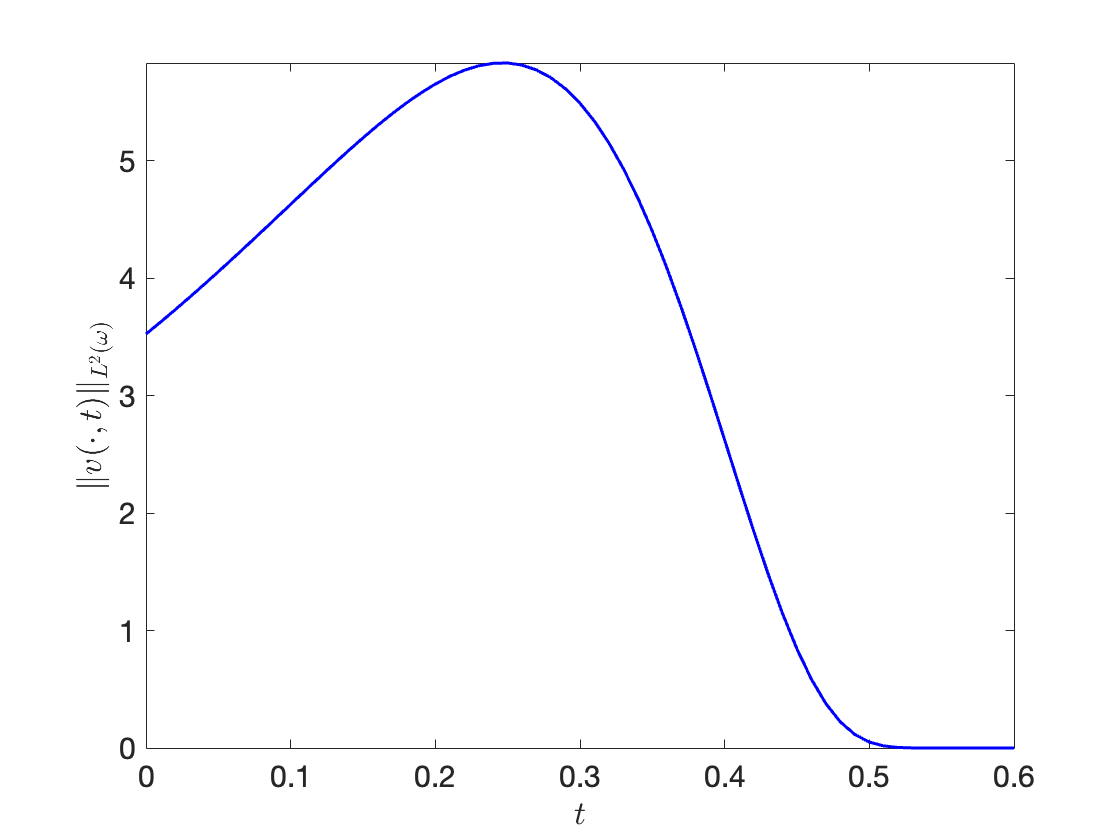}
    	\vskip-.5cm
	    \caption{The $L^2$-norm of $\bm{v}(\cdot,t)$}
    	\label{Fig:Stokes-vnorm}
    \end{minipage}    
\end{figure}

\begin{figure}[!h]
	\centering 
	\includegraphics[scale=0.22]{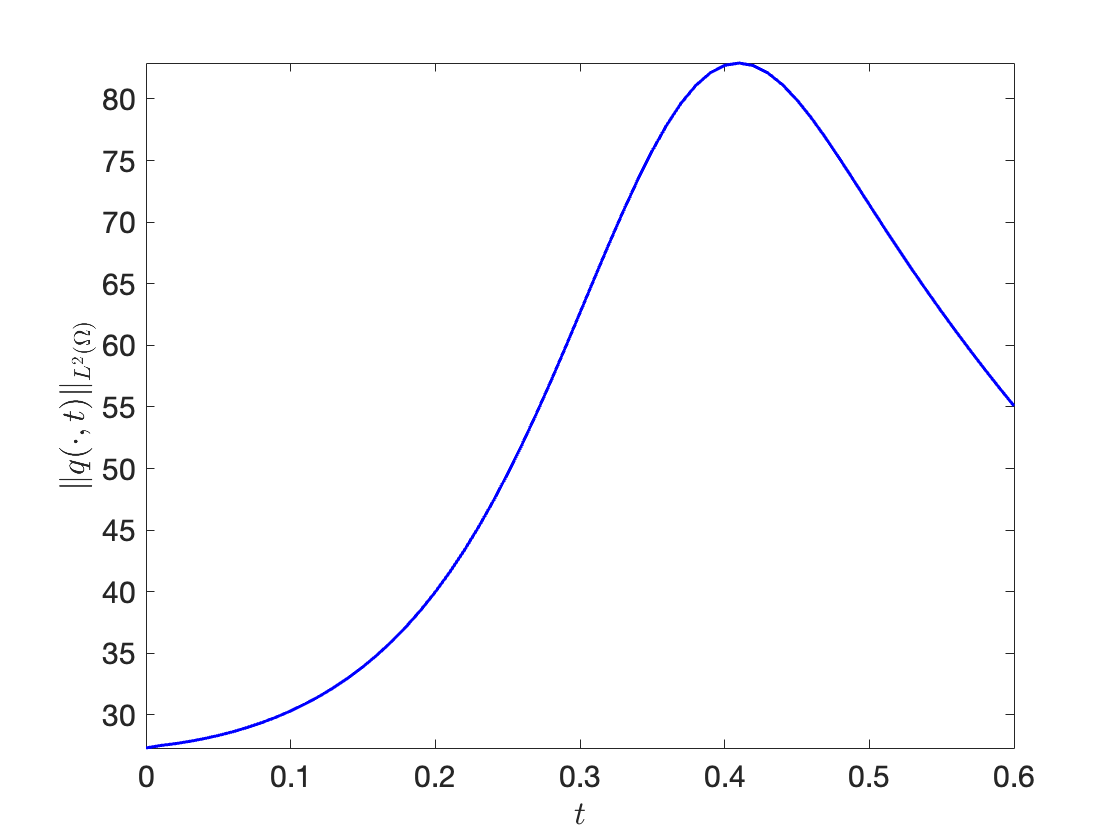}
	\vskip-.5cm
    \caption{$\|\bm{q}(\cdot\,,t)\|_{\bm{L}^2(\omega)}$.}
    \label{Fig:Stokes:qnorm}
\end{figure}

\section{Summary and further comments}
\label{Section:further:comments}

   In this paper, we have presented several algorithms based on Lagrangian and Augmented Lagrangian formulations of null controllability problems for the heat and the Stokes PDEs.
   To apply the techniques, we introduced a large parameter~$R>0$ in order to truncate the weight function associated with the state variable and then a second paremater~$K>0$ that plays the role of penalization in the Augmented Lagrangian.
   As~$R$ goes to~$+\infty$, we recover the solution of the original problem.
   This is proved rigorously and has been numerically validated in several experiments.
   
   One of the main virtues of the presented methods is that it can be applied with reasonable effort to control problems for high spatial dimension PDEs, which is not so clear for other strategies.
   
   The arguments and results can be extended and adapted to many other controllability problems. 
   Thus, they are valid for the internal or boundary control of parabolic PDEs and systems complemented with other boundary conditions, control problems in other domains, etc.
   
   In a next future, we will investigate their utility in the context of some semilinear and nonlinear problems like state-dependent diffusion heat equations, Burgers, Navier-Stokes, and Boussinesq systems, etc.
   
   The formulation and resolution of null or exact controllability problems with Lagrangian and Augmented Lagrangian methods for wave and Schr\"odinger PDEs remain, as far as we know, unexplored.
   This will also be investigated in forthcoming work.

\section*{Acknowledgments}
	The first and third authors were partially supported by Grant PID2020–114976GB–I00, funded by MCIN/AEI/10.13039/501100011033. The second author has been funded under the Grant QUALIFICA by Junta de Andaluc\'ia grant number QUAL21 005 USE.

\bibliography{biblio02}
\bibliographystyle{plain}

\end{document}